\documentclass[11pt]{article}

\usepackage[T1]{fontenc}
\usepackage{kpfonts}

\usepackage{amsmath,amssymb,amsthm}
\usepackage{mathrsfs}
\usepackage{dsfont}
\usepackage{esint}

\usepackage{array}
\usepackage{geometry}
\usepackage{graphicx}
\usepackage{xcolor}
\usepackage{enumitem}
\usepackage{appendix}
\usepackage{lineno}

\usepackage{changes}
\setdeletedmarkup{\textcolor{magenta}{\sout{#1}}}

\usepackage{hyperref}

\numberwithin{equation}{section}

\newtheorem{theo}{Theorem}[section]

\newtheorem{lem}[theo]{Lemma}
\newtheorem{cor}[theo]{Corollary}
\newtheorem{prop}[theo]{Proposition}
\theoremstyle{remark}
\newtheorem{rmk}[theo]{Remark}

\newtheorem{example}[theo]{Example}
\theoremstyle{definition}
\newtheorem{defi}[theo]{Definition}

\allowdisplaybreaks
\newcommand {\R} {\mathbb{R}}
\newcommand {\gam} {\gamma} 
\newcommand {\gamb} {\gamma^b} 
\newcommand {\tplan} {\pi} 
\newcommand {\phib} {F_b} 
\newcommand {\rhs} {f} 
\newcommand {\Rhs} {F} 
\newcommand {\QQ} {{\Omega\times \R^n}} 
\newcommand {\QQD} {{\Omega\times D}} 
\newcommand {\p} {\partial} %
\newcommand {\Hau} {\mathcal{H}_{d-1}} 
\newcommand {\nO} {n_\Omega} 
\newcommand {\q} {Q} 
\newcommand{\Dir}{\mathrm{Dir}} 
\newcommand{\Div}{\mathrm{div}} 
\newcommand{\EEu} {\mathcal E} 
\newcommand{\EE} {\mathrm E} 
\newcommand{\cM}{\mathcal{M}} 
\newcommand{\cP}{\mathcal{P}} 
\newcommand{\dd}{\mathrm{d}} 

\newcommand {\g} {f_b} 

\begin{document}


\title{Semilinear Neumann boundary value problems \\ for measure-valued maps}
\date{}
\author{Aleksei Kroshnin
	\thanks{Weierstrass Institute, Anton-Wilhelm-Amo-Strasse 39,  10117 Berlin, Germany.
		email: alex.kroshnin@gmail.com
	}
	\and Hugo Lavenant
	\thanks{Department of Decision Sciences and BIDSA, Bocconi University, Via Roentgen 1, 20136 Milano, Italy.
		email: hugo.lavenant@unibocconi.it
	} \and Dmitry Vorotnikov
	\thanks{University of Coimbra, CMUC, Department of Mathematics, 3001-501 Coimbra, Portugal.
		email: mitvorot@mat.uc.pt
	}
}

\maketitle

\begin{abstract}
We study the Wasserstein lift of nonlinear Poisson systems with inhomogeneous Neumann boundary conditions. We establish the existence of minimizers using new estimates involving the 
$p$-moments of measure-valued traces. We also derive the optimality conditions, which turn out to be matrix-valued generalizations of pressureless Euler equations. While deterministic minimizers exist under suitable convexity assumptions and for one-dimensional source domains, we show that genuinely measure-valued minimizers arise in higher dimensions. In particular, we construct examples in which the lifted and classical variational problems have different minimum values, thereby precluding the solutions to the lifted problem from being deterministic, despite the absence of any a priori mechanism enforcing measure-valued behavior. To the best of our knowledge, this phenomenon is new in the nonlinear elliptic theory. 
\end{abstract}

\vspace{1cm}

Keywords: optimal transport, lift, Wasserstein space, nonlinear Poisson system

\textbf{MSC [2020]: 35J20, 35J61, 49Q20, 49Q22, 53C23}

\section{Introduction}

The study of PDEs valued in Riemannian manifolds is a rich and fascinating area, bringing together analysis and geometry in a particularly subtle way \cite{Eells64, SU82, Struwe97, Lin2008, Tataru11}. From this viewpoint, even simple PDEs, such as the Laplace equation, acquire an intrinsically nonlinear character. The situation becomes even more intriguing when the smooth structure of a manifold is replaced by a general metric space, possibly endowed with additional geometric features, such as CAT structures \cite{KS93, Jost94, LMTV24, Mondino2022, Lin2026}.

On the other hand, the theory of Monge–Kantorovich optimal transport provides a remarkable example of a metric space. It is the Wasserstein space, whose points are probability measures. It can be viewed both as a geodesic space with non-negative Alexandrov curvature and, via Otto calculus \cite{Otto01}, as a formal infinite-dimensional Riemannian manifold \cite{Lott07}. This has fueled many advances across PDEs (in particular, in connection with gradient flows on this space), geometric analysis, probability, and data science, cf. \cite{Villani03,Villani09,S15,FiG21,Rig25}.

Even without invoking explicit Otto calculus or metric geometry, optimal transport provides a canonical framework for lifting mathematical objects (such as functionals or equations) to their blurred, measure-valued counterparts; see, for instance, \cite{HL23}. Notably, already more than 25 years ago, Brenier \cite{B01} proposed studying harmonic maps with values in the Wasserstein space. Such harmonic maps were studied in depth by the second author in \cite{HL19}. In particular, it was shown that Brenier’s notion of harmonic maps coincides with the abstract metric-geometric notion introduced in \cite{KS93,Jost94}. It was also proved in \cite{HL19,LMTV24} that those maps possess the so-called Ishihara property, meaning they pull back geodesically convex functions onto subharmonic ones.

In view of the above observations, a natural direction, which we pursue in this paper, is to move beyond the Laplace equation and study semilinear elliptic systems with a variational structure. We will view these systems as taking values, after lifting, in the Wasserstein space of probability measures $(\cP(\R^n),W_2)$, regarded as an extended metric space. For brevity, we will sometimes refer to our construction, based on optimal-transport geometry, as the Wasserstein lift.

Our additional motivations (in particular, for considering zero-order nonlinearities and inhomogeneous Neumann boundary conditions for our lifted problems) are as follows. Firstly, as noted by one of the authors in \cite[Section 6.5]{kroshnin}, very similar problems, subject to an additional ``transportation'' constraint, arise from a certain Kantorovich formulation of regularized optimal transport (with gradient penalization in the spirit of \cite{Louet16}). Secondly, when the source domain is one-dimensional (i.e., a segment), related problems involving a measure-valued Dirichlet boundary condition imposed at one endpoint and an inhomogeneous Neumann boundary condition at the other (sometimes referred to as ballistic optimal transport) have been investigated in \cite{CMP18,V22,A23,V25,MST26,V26}. They belong to a broad class of minimax problems whose underlying framework was originally developed by Brenier through the dual formulation of nonlinear PDEs, see Remark~\ref{brenier} below for a more detailed discussion. Moreover, Brenier \cite{Bvac} derived variational problems related to the latter ones, with a multidimensional source and values in matrix-valued measures, from the Einstein equations of general relativity in vacuum.

To keep the presentation simple, we will restrict our attention to the Wasserstein lift of the variational formulation of the  inhomogeneous Neumann problem for the semilinear Poisson system \begin{equation}
	\label{eq:EL_cl}
	\begin{cases}
		-\Delta u(x)=f(x,u(x)) & \text{in } \Omega\subseteq \R^d, \\
		\displaystyle{\frac {\p u }{\p \nO}(x)=\g(x,u(x))} & \text{on } \p \Omega,
	\end{cases}
\end{equation} without imposing any additional constraints. Of course, this class of models has various applications across disciplines such as physics, engineering, chemistry and biology, cf. \cite{Roub,MMM}. In the classical formulation, the unknown function $u(x)$ and the data $f$ and $\g$ take values in $\R^n$. In our measure-valued framework, the values of the unknown function are probability distributions on $\R^n$, with the Dirac lift $\delta_{u(x)}$ corresponding to the ``deterministic'' case. This measure-valued setting should not be confused with elliptic problems with rough data, in which $f$ is a measure on $\Omega$, see, e.g., \cite{MMM}.

Unlike the measure-valued Dirichlet problems studied in \cite{B01,HL19} and the problems with mixed, ballistic boundary conditions considered in \cite{CMP18,V22}, together with subsequent developments, the inhomogeneous Neumann problem contains no obvious a priori mechanism forcing the solution to be genuinely measure-valued. It is therefore natural to expect that the lifted problem might have deterministic optimizers. Our key finding is that this intuition fails as soon as the source domain $\Omega$ is not one-dimensional. In the purely Neumann setting, the lifted problem can be free of deterministic minimizers. In fact, we construct examples in which minimizers exist, but none can be deterministic. In particular, although a classical minimizing solution $u(x)$ also exists, the corresponding Dirac-valued map $\delta_{u(x)}$ is not a minimizer of the lifted problem.

In Section~\ref{s:remin}, we briefly recall some basic facts about measures, convergence and compactness. In Section~\ref{s:remmeas}, we describe the measure-valued maps that will be the objects of our study, together with their representation as measures on the product space. In Section~\ref{s:gke}, we introduce a generalized continuity equation and show that certain fiberwise averages of measure-valued maps are weakly differentiable, allowing us to handle variations of such maps. In Section~\ref{s:dene}, we present the definition of the Dirichlet energy for measure-valued maps based on the Wasserstein lift (\cite{HL19}, see also \cite{B01}), and discuss its basic properties. It was shown in \cite{HL19} that this definition is equivalent to the metric-geometric definitions of Korevaar--Schoen \cite{KS93} and Jost \cite{Jost94}, and is compatible with the theory developed by Reshetnyak \cite{Reshetnyak97,Reshetnyak04}. In Section~\ref{s:tracem}, we define the boundary trace of a measure-valued map with finite Dirichlet energy, characterize it through the continuity equation, and prove its stability under weak convergence. We also derive important estimates involving the $p$-moments of the traces, which are needed to establish the existence of our lifted Neumann problem. This problem, which is the main focus of our paper, is introduced in Section~\ref{clas-lif}, following a discussion of the classical variational formulation of~\eqref{eq:EL_cl}. In Section~\ref{s:exis}, we prove existence of a solution to the lifted problem. The main difficulty lies in establishing coercivity of the energy, whereas lower semicontinuity is comparatively more straightforward and follows from the lifted Carathéodory framework discussed in detail in Appendix~\ref{sec:appendix_caratheodory}. In Section~\ref{sec-opt}, we derive the optimality conditions~\eqref{eq:optimality_int_acc}, which can be viewed as a matrix-valued generalization of the pressureless Euler equations, also known as the sticky particle system \cite{ERS96,BG98,Villani03}. We also show that our measure-valued solutions satisfy a suitable renormalized formulation in the spirit of DiPerna–Lions \cite{DiPerna-Lions89,DiPernaLions89A}. In Section~\ref{rem:dip}, we compare our measure-valued solutions obtained via the Wasserstein lift with another notion of measure-valued solution inspired by DiPerna's ideas in \cite{DiPerna85}. We argue that our approach, including our way of ``computing the Laplacian'', is fundamentally more geometric and captures substantially more information in a finer and more precise manner. Moreover, our motivation differs from that of DiPerna's original work and subsequent developments concerning Young measure solutions, such as \cite{DiPernaM87,MNRR96}. \textbf{Our objective is not merely to solve the relatively simple PDE~\eqref{eq:EL_cl} in a relaxed sense. Rather, we seek to solve its Wasserstein lift, which, generally speaking, constitutes a genuinely different problem}.
However, in Section~\ref{s:consis}, we show that when the energy is convex in a suitable sense, passing from the classical problem to the lifted one does not yield any novelty. More precisely, in this case, whenever a classical solution exists, there is at least one deterministic solution of the lifted problem. The same holds when $\Omega$ is one-dimensional, as proved in Section~\ref{s:1d}. Finally, in Section~\ref{s:measurev}, as anticipated above, we provide two examples of genuinely measure-valued nature. In these examples, both the non-lifted minimization problem~\eqref{e:mainnotlifted} and the lifted minimization problem~\eqref{e:mainlifted} admit solutions, but their optimal values differ, hence the solutions of the lifted problem cannot be deterministic.

\section{Preliminaries} \label{pre}


\subsection{Reminders of measure theory and topology} \label{s:remin}

For a Polish space $\Theta$ (endowed with its Borel-$\sigma$ algebra), meaning a complete separable metric space, we denote by $\cM(\Theta)$ and $\cP(\Theta)$, respectively, the sets of finite signed measures and probability measures. If $\sigma \in \cM(\Theta)^n$ is a vectorial measure and $\| \cdot \|$ is a norm over $\R^n$, we denote by $\| \sigma \|$ the total variation measure; it is the unique non-negative scalar measure on $\Theta$ such that $\sigma = f \| \sigma \|$, with $f : \Theta \to \R^n$ a vector-valued function which is of unit norm. The push-forward of a measure $\sigma$ by a measurable map $f$ is denoted by $f_\# \sigma$, defined as $(f_\# \sigma)(A) = \sigma(f^{-1}(A))$ for any measurable set.

The notations $C_c(\Theta)$ and $C_b(\Theta)$ denote respectively the spaces of continuous compactly supported and continuous bounded functions. Let $C_0(\Theta)$ be the closure of $C_c(\Theta)$ in the strong topology of $C_b(\Theta)$. The narrow, or weak, topology on $\cM(\Theta)$ is the coarsest one making the map $\sigma \mapsto \int f \, \dd \sigma$ continuous for $f \in C_b(\Theta)$. The weak-$\star$ convergence is the coarsest one making the map $\sigma \mapsto \int f \, \dd \sigma$ continuous for $f \in C_0(\Theta)$.  Clearly from the definition, narrow convergence implies weak-$\star$ convergence.

We recall some useful criteria for relative compactness. First, a subset $\mathcal{K} \subseteq \cM(\Theta)$ is relatively compact for the topology of weak-$\star$ convergence if and only if $\sup_{\sigma \in \mathcal{K}} \| \sigma \|(\Theta) < + \infty$, cf. \cite[Theorem 1.59]{AFP}. For the space of probability measures with narrow convergence, we use the following characterization of tightness: a subset $\mathcal{K} \subseteq \cP(\Theta)$ is relatively compact for the topology of narrow convergence if and only if there exists a function $g : \Theta \to [0, + \infty]$ with compact sublevel sets such that 
$\sup_{\sigma \in \mathcal{K}} \int g(\theta) \dd \sigma(\theta) < + \infty$, \cite[Remark 5.1.5]{AmbrosioGigliSavare08}.

\subsection{Measure-valued maps}
\label{s:remmeas}

Let $\Omega\subseteq \R^d$ be an open bounded domain with Lipschitz boundary $\p \Omega$ and normalized to have unit volume. The unit outward normal at $x \in \p \Omega$ is $\nO(x)$. We write $\Hau$ for the $(d-1)$-dimensional surface measure on its boundary. We consider maps valued in the space of probability distributions: our main object of study will be $\gam$ a measurable map from $\Omega$ to $\cP(\R^n)$, only defined for a.e.\ $x \in \Omega$. We write $\gam \in L^0(\Omega, \cP(\R^n))$, and we denote by $\gam_x \in \cP(\R^n)$ the evaluation of $\gam$ at $x$. To every $\gam \in L^0(\Omega, \cP(\R^n))$ we can associate the measure $\pi = \dd \gam_x \dd x$ as a measure on $\Omega \times \R^n$, namely $\dd \pi(x,y) = \dd \gam_x(y) \dd x$. As $\Omega$ has unit volume, $\pi$ is a probability measure. It has first marginal $\dd x$ the Lebesgue measure on $\Omega$. Conversely, if $\pi$ is a probability measure on $\QQ$ whose first marginal is the Lebesgue measure, by the disintegration theorem \cite[Theorem 5.3.1]{AmbrosioGigliSavare08}, there exists a unique $\gam \in L^0(\Omega, \cP(\R^n))$ such that $\pi = \dd \gam_x \dd x$. 

If $(\gam_m)_{m \geq 1}$ is a sequence of maps in $L^0(\Omega,\cP(\R^n))$, we say that $\gamma_m$ converges weakly to $\gam \in L^0(\Omega,\cP(\R^n))$ if $(\gam_m)_x \dd x$, as a measure on $\QQ$, converges narrowly (i.e., in duality with $C_b(\QQ)$) to $\gam_x \dd x$.

We will also need maps valued in the space of matrix-valued measures: we denote them by $\q \in L^1(\Omega, \cM(\R^n)^{dn})$. Here it means that we have a measurable collection $(\q_x)_{x \in \Omega}$ of matrix-valued measures, such that the map $x \mapsto \| \q_x \|(\R^n)$ is integrable. Equivalently, it means that the measure $\dd \q_x \dd x$, as a matrix-valued measure on $\QQ$, has finite mass. For such a $\q$ and a function $\psi\in C_b(\QQ,\R^n)$, we abuse notation and write 
\[\int_{\R^n} \psi(x, y) \dd \q_x(y)\] 
instead of 
\[\int_{\R^n}
\psi(x, y) \dd \q^\top_x(y).\]

\subsection{The generalized continuity equation}
\label{s:gke}

We denote by $A : B = \mathrm{Tr}(A^\top B)$ the Frobenius inner product between two matrices, and $\| A \| = \sqrt{ A : A}$ the Frobenius norm. 

\begin{defi}
	Let $\gam \in L^0(\Omega, \cP(\R^n))$ and $\q \in L^1(\Omega, \cM(\R^n)^{dn})$. We say that the pair $(\gam,\q)$ satisfies in a weak sense the generalized continuity equation
	\begin{equation}
		\label{eq:CE}
		\nabla_x \gam + \Div_y \q = 0
	\end{equation}
	if and only if, for every test function $\varphi \in C^1_c(\QQ)^d$, we have 
	\begin{equation}
		\label{eq:CE_weak}
		\iint_\QQ \Div_x \varphi \, \dd \gam_x \dd x + \iint_\QQ \nabla_y \varphi : \dd \q_x \, \dd x = 0.
	\end{equation}
\end{defi}

As $\dd\gam_x \dd x$ and $\dd \q_x \dd x$ are finite measures, by a truncation argument it is easy to see that~\eqref{eq:CE_weak} is still valid if $\varphi \in C^1_b(\QQ)$, the space of functions over $\QQ$ which are bounded and with continuous bounded derivative. 

We report a direct corollary of this generalized continuity equation: for a test function $\varphi$, the function $x \mapsto \int_{\R^n} \varphi(x,y) \dd \gam_x (y)$ can be differentiated weakly. We only state a result without the sharpest growth assumption, and we refer to \cite[Chapter 6]{kroshnin} for extensions.

\begin{lem}\label{lem:weak_der_int}
	Let $\gam \in L^0(\Omega, \cP(\R^n))$ and $\q \in L^1(\Omega, \cM(\R^n)^{dn})$ satisfy the generalized continuity equation~\eqref{eq:CE}. 
	Fix $\varphi \in C^1_b(\QQ)$ bounded and with bounded continuous derivative. Define the function $g$ as, for $x \in \Omega$,
	\[
	g(x) = \int_{\R^n} \varphi(x, y) \dd \gam_x(y). 
	\]
	Then $g$ belongs to the space $(W^{1,1}\cap L^\infty)(\Omega)$ with derivative
	\begin{equation}\label{eq:weak_der_int}
		\nabla g(x) = \int_{\R^n} \nabla_x \varphi(x, y) \, \dd \gam_x(y) + \int_{\R^n} \nabla_y \varphi(x, y) \dd \q_x(y).
	\end{equation}
\end{lem}

\begin{proof}
	Fix an arbitrary $\eta \in C_c^1(\Omega)^d$. Then
	\begin{align*}
		\int_\Omega g(x) \Div_x \eta(x) \dd x &= \iint_{\QQ} \varphi(x,y) \Div_x \eta(x)  \dd \gam_x \dd x \\
		& = \iint_{\QQ}  \Div_x ( \varphi(x,y) \eta(x))   \dd \gam_x \dd x - \iint_{\QQ}  \nabla_x \varphi(x,y) \cdot \eta(x)  \dd \gam_x(y) \dd x \\
		& = - \int_{\Omega} \eta(x) \cdot \left( \int_{\R^n} \nabla_y \varphi(x,y)  \dd \q_x(y) + \int_{\R^n} \nabla_x \varphi(x,y)   \dd \gam_x(y) \right) \dd x,
	\end{align*}
	where the last equality follows from the continuity equation with test function $\eta(x) \varphi(x,y)$.
	Therefore, $g$ is weakly differentiable with the weak derivative given by~\eqref{eq:weak_der_int}. The fact that $g$ and $\nabla g$ belong to $L^\infty(\Omega)$ and $L^1(\Omega)$, resp., follows from the boundedness assumptions on $\varphi$.
\end{proof}

\subsection{The Dirichlet energy} \label{s:dene}

Recall that $\| \cdot \|$ is the Frobenius norm for matrices.

\begin{defi}
	If $\gam \in L^0(\Omega, \cP(\R^n))$, we define
	\begin{equation}
		\label{eq:def_Dirichlet}
		\Dir(\gam) = \inf_\q \int_\Omega \left( \int_{\R^n} \left\|\frac{\dd \q_x}{\dd \gam_x}(y) \right\|^2 \dd \gam_x(y)  \right) \dd x, 
	\end{equation}
	where the infimum is taken over all measurable maps $\q \in L^1(\Omega,\cM(\R^n)^{dn})$ such that $\q_x \ll \gam_x$ for a.e.\ $x$, and such that the equation~\eqref{eq:CE} is satisfied in a weak sense. 
\end{defi}

\begin{prop}[{\cite[Lemma 4.6]{HL23}}]
	If $\Dir(\gam) < + \infty$ then the infimum in~\eqref{eq:def_Dirichlet} is attained, for a field $(\q_x)_x$ which is unique $x$-a.e.
\end{prop}

\noindent We call the minimal $\q$ the \emph{tangent} flux to $\gam$.

\begin{rmk}
	\label{rem:tangent_gradient}
	Calling $v_x \coloneqq \frac{\dd \q_x}{\dd \gam_x}$, where $\q$ is the tangent flux, we can characterize $(x,y) \mapsto v_x(y)$ as a field that belongs to the closure of the set
	$\{ \nabla_y \varphi(x,y) \ : \ \varphi \in C^1_c(\QQ, \R^d) \}$
	in the space $L^2(\QQ, \dd \gam_x \dd x)^{dn}$: see \cite[Proposition 3.11]{HL19}.
\end{rmk}

\begin{rmk}
	If $\gam_x$ is supported in a fixed compact set independent of $x \in \Omega$, \cite[Theorem 3.26]{HL19} shows that $\Dir$ coincides with the Dirichlet energy, in the metric-geometric sense of \cite{KS93,Jost94}, of $\gam$ seen as a map valued in the Wasserstein metric space $(\mathcal{P}_2(\R^n),W_2)$. In the present study, we do not impose any moment condition on $\gam_x$, so our map rather takes values in the extended Wasserstein space $(\mathcal{P}(\R^n),W_2)$, cf. \cite{Lisini16,Lisini07}, where the distance $W_2$ may possibly be infinite. We will not further expand on the metric point of view, as it does not help us prove existence or derive the optimality conditions. Nevertheless, these facts, together with the lifting identity~\eqref{eq:lifting_identity} below, explain why we call $\Dir$ the ``Wasserstein lift'' of the classical Dirichlet energy.
\end{rmk}

\noindent Eventually we report here that $\Dir$ is a suitable functional to be minimized in the sense of calculus of variations.

\begin{theo}[{\cite[Proposition 4.7]{HL23}}]\label{t:dir}
	The functional $\Dir$ is convex and lower semi-continuous with respect to weak convergence.
\end{theo}

\subsection{Trace of measure-valued maps} \label{s:tracem}

We move to the definition of the (boundary) trace of a measure-valued map. The existence and stability of the trace was already proved in \cite{HL19} under an additional compactness assumption, and the present proof follows exactly the same path. We also present new trace estimates which will be crucial to prove well-posedness of our problem. 

\begin{theo}
	\label{theo:boundary_values}
	Assume $\gam \in L^0(\Omega,\cP(\R^n))$ and $\Dir(\gamma) < + \infty$. Then there exists a measure-valued map $\gamb \in L^0(\p \Omega, \cP(\R^n))$, uniquely determined by $\gam$, such that the following holds: for any $\q \in L^1(\Omega,\cM(\R^n)^{dn})$ such that~\eqref{eq:CE} holds, for all $\varphi \in C^1_b(\bar{\Omega} \times \R^n)^d$ 
	\begin{equation}
		\label{eq:CE_with_b}
		\iint_\QQ \Div_x \varphi \, \dd \gam_x \dd x + \iint_\QQ \nabla_y \varphi : \dd \q_x \, \dd x = \int_{\p \Omega}  \int_{\R^n} \varphi(x,y) \cdot \nO(x) \, \dd \gamb_x(y) \dd \Hau(x).
	\end{equation}
	In addition, if $(\gam_m)_{m \geq 1}$ is a sequence in $L^0(\Omega, \cP(\R^n))$ converging weakly to $\gamma$ and with $\Dir(\gam_m)$ uniformly bounded in $m$, then $(\gamma_m)^b_x \dd \Hau(x)$ converges weakly to $\gamb_x \dd \Hau(x)$ as $m \to + \infty$, in duality with $C_b(\p \Omega \times \R^n)$. 
\end{theo}

\begin{proof}
	\emph{1st step: existence of $\gamb$.}
	Take $\q$ such that the continuity equation~\eqref{eq:CE} is satisfied, as $\Dir(\q) < + \infty$ there exists at least one such $\q$. 
	Let us define $T(\varphi)$, for $\varphi \in C^1_b(\bar{\Omega} \times \R^n)^d$, as the left-hand side of~\eqref{eq:CE_with_b}. It defines $T$ as a distribution over $\bar{\Omega} \times \R^n$. 
	Now fix $\varphi \in C^1_b(\bar{\Omega} \times \R^n)^d$, and consider, for $x \in \Omega$, 
	\begin{equation*}
		g(x) = \int_{\R^n} \varphi(x,y) \dd \gam_x(y).
	\end{equation*}
	By Lemma~\ref{lem:weak_der_int}, we have $g \in (W^{1,1}\cap L^\infty)(\Omega)^d$. In particular, there exists $\bar{g} \in L^\infty(\p \Omega)^d$ that coincides with the trace of $g$ a.e.\ on $\p \Omega$, cf. \cite[Section 4.3]{EG}. From the explicit expression of $\nabla g$ in~\eqref{eq:weak_der_int} and the weak formulation of the trace, we have
	\begin{equation*}
		T(\varphi) = \int_{\Omega} \Div_x g(x) \dd x = \int_{\p \Omega} \bar{g}(x) \cdot n_\Omega(x) \dd \Hau(x). 
	\end{equation*}
	From this we see that $T$ does not depend on the particular choice of $\q$, as long as the continuity equation is satisfied. Moreover, as $\| \bar{g} \|_\infty \leq \| g \|_\infty \leq \| \varphi \|_\infty$, we obtain that $|T(\varphi)| \leq \Hau(\p \Omega) \| \varphi \|_\infty$. That is, $T$ is a distribution of order $0$: it can be represented by a finite vectorial measure $\sigma \in \cM(\QQ)^d$. From the explicit representation of the trace as $\bar{g}(x) = \lim_{r \to 0} \fint_{B_r(x) \cap \Omega} g(x') \dd x'$ for a.e.\ $x \in \p \Omega$, we deduce that $\sigma$ is supported on $\p \Omega \times \R^n$, with direction given by $n_\Omega$, that is, $\sigma = n_\Omega \cdot \tilde{\sigma}$ for a scalar measure $\tilde{\sigma}$ on $\p \Omega \times \R^n$. Testing with functions $\varphi$ such that $\varphi \cdot n_\Omega \geq 0$ on $\p \QQ$, we obtain $\bar{g} \cdot n_\Omega \geq 0$ and thus $T(\varphi) \geq 0$, which implies that $\tilde{\sigma}$ is a non-negative measure. Moreover, using $\varphi = \varphi(x)$ of class $C^1$ over $\bar{\Omega}$ depending only on $x$, so that $g = \varphi$ (and thus $\bar{g} = \varphi$), we obtain that the marginal distribution of $\tilde{\sigma}$ over $\p \Omega$ is $\Hau$. In conclusion, by the disintegration theorem there exists $(\gamb_x)_{x \in \p \Omega}$ a family of probability distributions over $\cP(\R^n)$ such that $\tilde{\sigma} = \dd \gamb_x \dd \Hau(x)$. The conclusion~\eqref{eq:CE_with_b} follows.
	
	\emph{2nd step: stability of the boundary conditions.} Assume that $(\gam_m)_{m \geq 1}$ is a sequence as in the theorem and take $\q_m$ the tangent flux. From the Cauchy-Schwarz inequality, we have $\| \dd (\q_m)_x \dd x\|(\QQ) \leq \sqrt{\Dir(\gam_m)}$, in particular it is bounded independently of $m$. By the criterion for relative compactness in $\cM(\QQ)$ \cite[Theorem 1.59]{AFP}, up to extraction of a subsequence $\dd (\q_m)_x \dd x$ converges weakly-$\star$ to some limit $\nu$. 
	The limit has finite total mass, and we claim that its first marginal must be absolutely continuous with respect to the Lebesgue measure, and
	thus by disintegration $\dd \nu=\dd Q_x \dd x$ for some map $\q \in L^1(\Omega,\cM(\R^n)^{dn})$. 
	
	To prove the claim, using the Cauchy--Schwarz inequality, for any open set $A \subset \Omega$ we estimate
	\begin{align*}
		\|\dd(Q_m)_x\dd x\|(A\times \R^n) & =  \int_A \left(\int_{\R^n} \left\|\frac{\dd \q_x}{\dd \gam_x}\right\| \dd \gam_x \right) \dd x \\
		& \leq \sqrt{|A|} \sqrt{\int_\Omega \left(\int_{\R^n} \left\|\frac{\dd \q_x}{\dd \gam_x}\right\|^2 \dd \gam_x\right) \, \dd x} = \sqrt{|A|} \sqrt{\Dir(\gam_m)}.
	\end{align*}
	By lower semicontinuity of the total variation with respect to weak-$\star$ convergence, we obtain
	$$
	\|\nu\|(A\times \R^n)
	\leq \liminf_{m\to+\infty}\|\dd(Q_m)_x\dd x\|(A\times \R^n)
	\leq C|A|^{\frac12}.
	$$
	This immediately yields the claim above.


	Passing to the limit, we see that $(\gam,\q)$ satisfies the equation~\eqref{eq:CE} weakly. Moreover, passing to the limit in the expression of the distribution $T$ as above, we see that for any $\varphi \in C^1_c(\bar{\Omega} \times \R^n)^d$, 
	\begin{equation*}
		\lim_{m \to + \infty} \iint_{\p \Omega \times \R^n} \varphi \cdot \nO \, \dd (\gamma_m)^b_x \dd \Hau(x) = \iint_{\p \Omega \times \R^n} \varphi \cdot \nO \, \dd \gamb_x \dd \Hau(x)
	\end{equation*}
	As any $\psi \in C^1_c(\p \Omega \times \R^n)$ can be represented as the trace $\varphi \cdot \nO$ for some $\varphi \in C^1_c(\bar{\Omega} \times \R^n)^d$,
	\begin{equation*}
		\lim_{m \to + \infty} \iint_{\p \Omega \times \R^n} \psi \, \dd (\gamma_m)^b_x \dd \Hau(x) = \iint_{\p \Omega \times \R^n} \psi \, \dd \gamb_x \dd \Hau(x).  
	\end{equation*}
	As all objects are probability measures, this convergence in duality with $C^1_c$ can be updated to a convergence in duality with $C_c$. 
	Eventually, as we know that $(\gamma_m)^b_x \dd \Hau(x)$ has the same total mass at the limit as $\gamb_x \dd \Hau(x)$, convergence in duality with $C_c$ can be improved to weak convergence, in duality with $C_b$, see \cite[Remark 5.1.6]{AmbrosioGigliSavare08}. 
\end{proof}

\noindent We deduce an immediate strengthening of Lemma~\ref{lem:weak_der_int} which is actually the way $\gamb$ was defined.

\begin{lem}
	\label{lem:weak_der_int_b}
	Under the same notations and assumptions of Lemma~\ref{lem:weak_der_int}, the function $g \in (W^{1,1}\cap L^\infty)(\Omega)$ has trace values given by, for a.e.\ $x \in \p \Omega$, 
	\begin{equation*}
		g(x) = \int_{\R^n} \varphi(x,y) \dd \gamb_x(y). 
	\end{equation*}
\end{lem}

\begin{proof}
	The same arguments in the proof of Lemma~\ref{lem:weak_der_int}, combined with the equation~\eqref{eq:CE_with_b}, yield: for any $\eta \in C^1_b(\bar{\Omega})$
	\begin{equation*}
		\int_\Omega g(x) \Div_x \eta(x) \dd x = - \int_\Omega \eta(x) \nabla g(x) \dd x + \int_{\p \Omega} (\eta(x) \cdot n_\Omega(x)) \left( \int_{\R^n} \varphi(x,y) \dd \gamb_x(y) \right) \dd \Hau(x).
	\end{equation*}
	The conclusion follows. 
\end{proof}

We move to useful trace estimates which will be crucial to get coercivity of our lifted functional.
We recall the classical trace estimates: there exists $C$, depending on $\Omega$, such that for all $g \in W^{1,1}(\Omega)$, using also $g$ to denote the trace of $g$ on $\p \Omega$ 
\begin{align}
	\label{eq:trace_int_to_b}
	\| g \|_{L^1(\p \Omega)} & \leq C (\| g \|_{L^1(\Omega)} + \| \nabla g \|_{L^1(\Omega)}), \\
	\label{eq:trace_b_to_int}
	\| g \|_{L^1(\Omega)} & \leq C (\| g \|_{L^1(\p \Omega)} + \| \nabla g \|_{L^1(\Omega)} ).
\end{align}
The first one is nothing else than the continuity of the trace operator from $W^{1,1}(\Omega)$ to $L^1(\p \Omega)$ \cite[Section 4.3]{EG}. The second one is also well-known, see, e.g., \cite{B_cons}.

We show that similar estimates hold for our measure-valued maps. We introduce a notation for the moments in the $y$ variable: for a measure space $(\Theta,\sigma)$ which will be $(\Omega,\dd x)$ or $(\p \Omega, \Hau)$ and a map $\gam \in L^0(\Theta, \cP(\R^n))$, we write 
\begin{equation}
	\label{eq:def_moment}
	M_{p,\Theta}(\gam) = \iint_{\Theta \times \R^n} |y|^p \dd \gam_x(y) \dd \sigma(x),
\end{equation}
where here $p > 0$ is an arbitrary exponent.

\begin{prop}
	\label{prop:trace_moments}
	Fix $0 < p \leq 2$. Then there exists a constant $C$ depending on $\Omega$ and $p$ such that, for any $\gam \in L^0(\Omega,\cP(\R^n))$ with $\Dir(\gamma) < + \infty$,
	\begin{align}
		\label{eq:trace_lift_int_to_b}
		M_{p, \p \Omega}(\gamb) & \leq C( M_{p,\Omega}(\gam) + \Dir(\gamma)^{p/2}), \\
		\label{eq:trace_lift_b_to_int} 
		M_{p, \Omega}(\gam) & \leq C( M_{p, \p \Omega}(\gamb) + \Dir(\gamma)^{p/2}).
	\end{align}
\end{prop}

\begin{proof}
	For any $m > 0$ consider a monotone concave function $\eta_m \in C^1_b([0, \infty))$ such that $\eta(t) = t$ for $t \in [0, m]$ and $\eta(t) \equiv 2 m$ for any $t > 3 m$.
	For $\epsilon > 0$, define the function $\varphi_{\epsilon,m}(y) = \eta_m\left((|y|^2 + \epsilon^2)^{p/2}\right)$. Clearly, $\varphi_{\epsilon,m} \in C^1_b(\R^n)$. Now, define 
	\begin{equation*}
		g_{\epsilon,m} : x \mapsto \int_{\R^n} \varphi_{\epsilon,m}(y) \, \dd \gam_x(y).
	\end{equation*}
	as in Lemma~\ref{lem:weak_der_int}. By Lemma~\ref{lem:weak_der_int} and Lemma~\ref{lem:weak_der_int_b}, $g_{\epsilon,m} \in (W^{1,1}\cap L^\infty)(\Omega)$, with trace given by $g_{\epsilon,m}(x)= \int_{\R^n} \varphi(y) \, \dd \gamb_x(y)$ for $x \in \p \Omega$, and with gradient 
	\begin{equation}
		\label{eq:g_gradient_only_on_y_a}
		\nabla g_{\epsilon,m}(x) = \int_{\R^n} \nabla \varphi_{\epsilon,m}(y) \dd \q_x(y). 
	\end{equation}
	
	\emph{1st case: $1 < p \leq 2$.}
	Since $0 \leq \eta_m'(t) \le \frac{\eta_m(t)}{t} \leq 1$ by concavity, and $0 < \frac{p-1}{p} \le 1$, we obtain
	\begin{align*}
		|\nabla \varphi_{\epsilon,m}(y)| &= p |y| (|y|^2 + \epsilon^2)^{p/2 - 1} \eta_m'\left((|y|^2 + \epsilon^2)^{p/2}\right) \\
		&\leq p (|y|^2 + \epsilon^2)^{(p-1)/2} \left(\frac{\varphi_{\epsilon,m}(y)}{(|y|^2 + \epsilon^2)^{p/2}}\right)^{(p-1)/p} \\
		&= p \varphi_{\epsilon,m}(y)^{(p-1)/p} .
	\end{align*}
	Starting from~\eqref{eq:g_gradient_only_on_y_a}, we apply the Cauchy--Schwarz, Young's, and the Minkowski inequalities to get for any $a > 0$
	\begin{align*}
		\int_\Omega |\nabla g_{\epsilon,m}(x)| \, \dd x &\leq \sqrt{\Dir(\gam)} \sqrt{ \iint_{\QQ} |\nabla \varphi_{\epsilon,m}(y)|^2 \, \dd \gam_x(y) \dd x } \\
		&\leq p \sqrt{\Dir(\gam)} \sqrt{ \iint_{\QQ} \varphi_{\epsilon,m}(y)^{\frac{2 (p-1)}{p}} \, \dd \gam_x(y) \dd x } \\
		&\leq a^{p-1} \Dir(\gam)^{p/2} + \frac{p-1}{a} \left(\iint_{\QQ} \varphi_{\epsilon,m}(y)^{\frac{2 (p-1)}{p}} \, \dd \gam_x(y) \dd x \right)^{\frac{p}{2 (p-1)}} \\
		&\leq a^{p-1} \Dir(\gam)^{p/2} + \frac{p-1}{a} \iint_{\QQ} \varphi_{\epsilon,m}(y) \, \dd \gam_x(y) \dd x .
	\end{align*}
	Combining with the classical estimate~\eqref{eq:trace_int_to_b}, and remembering that $\frac p 2 \le 1$, we obtain
	\begin{equation*}
		\iint_{\p \Omega \times \R^n} \varphi_{\epsilon,m}(y) \, \dd \gamb_x(y) \dd \Hau (x) \leq C \left(\Dir(\gam)^{p/2} + \iint_{\QQ} |y|^p \, \dd \gam_x(y) \dd x + \epsilon^p\right).
	\end{equation*}
	Letting $\epsilon \to 0$ and $m \to \infty$, Fatou's lemma yields~\eqref{eq:trace_lift_int_to_b}.
	On the other hand, to obtain~\eqref{eq:trace_lift_b_to_int}, we start from the classical estimate~\eqref{eq:trace_b_to_int} for the function $g_{\epsilon,m}$ and obtain
	\begin{align*}
		&\iint_{\QQ} \varphi_{\epsilon,m}(y) \, \dd \gam_x(y) \dd x \leq C \left(\int_\Omega |\nabla g_{\epsilon,m}(x)| \, \dd x + \iint_{\p \Omega} g_{\epsilon,m}(y) \, \dd \Hau (x)\right)\\
		&\qquad \leq C \left(a^{p-1} \Dir(\gam)^{p/2} + \frac{1}{a} \iint_{\QQ} \varphi_{\epsilon,m}(y) \, \dd \gam_x(y) \dd x + \iint_{\p \Omega \times \R^n} \varphi_{\epsilon,m}(y) \, \dd \gamb_x(y) \dd \Hau (x) \right) .
	\end{align*}
	For $a$ large enough, the term $\iint_{\QQ} \varphi_{\epsilon,m}(y) \, \dd \gam_x(y) \dd x = \int g_{\epsilon,m}(x) \dd x$ in the right-hand side is absorbed by the left-hand side. Again, letting $\epsilon \to 0$, $m \to \infty$, one gets~\eqref{eq:trace_lift_b_to_int} by Fatou's lemma.
	
	\emph{2nd case: $0 < p \leq 1$.} In this case, elementary algebra shows that
	\begin{equation*}
		|\nabla \varphi_{\epsilon,m}(y)| \leq p |y| (|y|^2 + \epsilon^2)^{p/2 - 1} \leq p \epsilon^{p - 1} .
	\end{equation*}
	First, consider the case $\Dir(\gamma) > 0$ and fix $\epsilon = \sqrt{\Dir(\gam)}$. Using the Cauchy--Schwarz inequality again, we get
	\begin{equation*}
		\int_\Omega |\nabla g_{\epsilon,m}(x)| \, \dd x \leq p \epsilon^{p - 1} \sqrt{\Dir(\gam)} = p \Dir(\gam)^{p/2} .
	\end{equation*}
	The rest follows directly from~\eqref{eq:trace_int_to_b} and~\eqref{eq:trace_b_to_int} using Fatou's lemma with $m \to \infty$ while keeping $\epsilon$ fixed.
	
	Finally, if $\Dir(\gamma) = 0$, then $\int_\Omega |\nabla g_{\epsilon,m}(x)| = 0$ and we conclude by taking $\epsilon \to 0$, $m \to \infty$.
	
\end{proof}

\begin{rmk}
	Take $\sigma \in \cP(\R^n)$ a fixed probability measure, and consider a constant measure-valued map $\gam : x \mapsto \sigma$. Then $\q \equiv 0$ is always such that $(\gam,\q)$ satisfies~\eqref{eq:CE}, meaning $\Dir(\gam) = 0$ and also $\gamb_x = \sigma$ for a.e.\ $x \in \p \Omega$. This holds without any moment condition on $\sigma$. In this case, assuming for simplicity that $\Omega$ and $\p \Omega$ have both unit measure, we have $M_{p,\Omega}(\gam) = M_{p,\p\Omega}(\gamb)$ for any $p > 0$, where both terms are potentially infinite.
	We deduce in particular that the exponents in the left hand side of~\eqref{eq:trace_int_to_b} and~\eqref{eq:trace_lift_b_to_int} cannot be improved, meaning we cannot find an estimate of the type $M_{q, \p \Omega}(\gamb)^{p/q} \leq C( M_{p,\Omega}(\gam) + \Dir(\gamma)^{p/2})$ for $q > p$. Said differently, there is no gain in integrability from a control of the Dirichlet energy, contrary to what happens for classical Sobolev functions.
\end{rmk}


\section{The measure-valued Neumann problem} \label{s:nneum}

Before stating our problem, in order to ensure that the integrals are well-defined and to define coercivity we need some vocabulary. In the following we consider $(\Theta, \sigma)$ a finite measure space, that will be $(\Omega, \dd x)$ or $(\p \Omega, \Hau)$.

\begin{defi}
	\label{defi:p_bounded_coercive}
	Given $p\ge 0$, a function $h : \Theta \times \R^n \to \R$ is said to:
	\begin{enumerate}
		\item be $p$-bounded from below if there exists $a \in L^1(\Theta,\sigma)$ and $C < +  \infty$ such that, for all $(\theta,y) \in \Theta \times \R^n$, 
		\begin{equation*}
			h(\theta,y) \geq a(x) - C |y|^p; 
		\end{equation*}
		\item be $p$-coercive if there exists $a \in L^1(\Theta,\sigma)$ and $c > 0$ such that, for all $(\theta,y) \in \Theta \times \R^n$, 
		\begin{equation*}
			h(\theta,y) \geq a(x) + c |y|^p.
		\end{equation*}
		\item have at most $p$-growth if there exists $a \in L^1(\Theta,\sigma)$ and $C < +  \infty$ such that, for all $(\theta,y) \in \Theta \times \R^n$, 
		\begin{equation*}
			|h(\theta,y)| \leq a(x) + C |y|^p; 
		\end{equation*}
	\end{enumerate}
\end{defi} 

\noindent Note that being $0$-bounded from below simply means, with our vocabulary, being bounded from below by an integrable function.

\subsection{The classical problem and its lifted version} \label{clas-lif}

Take $\Rhs$, $\phib$ two measurable functions from respectively $\Omega \times \R^n$ and $\p \Omega \times \R^n$ to $\R$. We first consider the following classical problem of calculus of variations:
\begin{equation}\label{e:mainnotlifted}  
	\min_{u\in H^1(\Omega,\R^n)} \EEu(u),
\end{equation} 
where
\begin{equation}\label{e:mainnotliftedenergy}  
	\EEu(u) \coloneqq \frac 1 2\int_\Omega |\nabla u(x)|^2\,\dd x+\int_\Omega \Rhs(x,u(x))\,\dd x+ \int_{\p \Omega} \phib(x,u(x))\,\dd \Hau(x).
\end{equation}
To be precise, we restrict to competitors $u\in H^1(\Omega,\R^n)$, and such that the negative parts $\Rhs(x,u)_-$ and $\phib(x,u)_-$ are integrable (if not, we set $\EEu(u)=+ \infty$ by convention). Thus, $\EEu(u)$ is well-defined as an element of $(- \infty, + \infty]$.


Under suitable growth and coercivity conditions, this problem has a solution and the Euler-Lagrange equations read as a semilinear elliptic system. Specifically, we call the partial derivatives
$$
-\nabla_y\Rhs(x,y) \eqqcolon \rhs(x,y), \qquad -\nabla_y\phib(x,y)=:\g(x,y).
$$
With these notations, the Euler–Lagrange equations are given by the semilinear elliptic system \eqref{eq:EL_cl}, which we recall here:
\begin{equation}
	\label{eq:EL_classical}
	\begin{cases}
		-\Delta u(x)=f(x,u(x)) & \text{in } \Omega, \\
		\displaystyle{\frac {\p u }{\p \nO}(x)=\g(x,u(x))} & \text{on } \p \Omega.
	\end{cases}
\end{equation}
The case $\Rhs$ independent of $y$ leads to $f \equiv 0$, thus to a harmonic solution $u$ inside of $\Omega$. If $\phib$ does not depend on $y$, then we have homogeneous Neumann boundary conditions. When $\phib$ depends on $y$, we obtain (generalized) Robin boundary conditions, the prototypical example being $\phib(x,y) = |y|^2/2$.

We propose to lift~\eqref{e:mainnotlifted} to measure-valued maps: resonating with the previous works \cite{HL19,HL23} we consider
\begin{equation}\label{e:mainlifted}
	\min_{\gam\in L^0(\Omega,  \cP(\R^n))} \EE(\gam),
\end{equation} 
where
\begin{equation}\label{e:mainliftedenergy}\EE(\gam) \coloneqq \frac 1 2 \Dir(\gam) +\iint_\QQ \Rhs(x,y) \,\dd\gam_x(y)\,\dd x+ \iint_{\p \QQ} \phib(x,y)\,\dd\gamb_x(y) \,\dd\Hau(x).
\end{equation}
This lift satisfies the \emph{lifting identity} in the sense of \cite{HL23}: given a classical map $u \in H^1(\Omega,\R^n)$ and its Dirac lift $\gam_u$, defined by
\[
(\gam_u)_x = \delta_{u(x)}, \qquad x \in \Omega,
\]
we have
\begin{equation}
	\label{eq:lifting_identity}
	\EE(\gam_u) = \EEu(u),
\end{equation}
see in particular \cite[Corollary 4.9]{HL23}.
Similarly to the deterministic case, we adopt the following convention for the definition: $\EE(\gam) = + \infty$ as soon as $\Dir(\gam) = + \infty$. Moreover, if $\Dir(\gam) < + \infty$, we only allow competitors such that $\Rhs_-$ and $(\phib)_-$ (the negative parts) are integrable with respect to respectively $\dd \gam_x \dd x$ and $\dd \gam_x \dd \Hau(x)$ (if $\gam$ does not satisfy the latter condition, we set $\EE(\gam) = + \infty$). With this convention, $\EE(\gam)$ is well-defined as an element of $(- \infty, + \infty]$.


\subsection{Existence} \label{s:exis}

We now state the existence of a solution to the lifted problem \eqref{e:mainlifted}. For this we apply the direct method of calculus of variations. The delicate part is to prove coercivity of $\EE$.
The intuition is that it comes from the Dirichlet energy together with either $\Rhs$ or $\phib$, while the other term, not being coercive, may even be unbounded from below, as long as it is controlled by the coercive term.

Lower-semi continuity on the other hand is much more standard to obtain: it relies on the notion of Carathéodory functions, see Appendix~\ref{sec:appendix_caratheodory} for the definitions and some standard results. 

\begin{theo}
	\label{theo:existence_lifted}
	Let $\Rhs$ and $\phib$ be a pair of Carath\'eodory functions. Assume that for some $0 \leq p < q \leq 2$, one of the following holds:
	\begin{enumerate}
		\item $\Rhs$ is $q$-coercive and $\phib$ is $p$-bounded from below;
		\item $\phib$ is $q$-coercive and $\Rhs$ is $p$-bounded from below.
	\end{enumerate}
	Assume also that there exists a function $u\in H^1(\Omega,\R^n)$ with finite energy $\EEu(u)$. Then the lifted problem \eqref{e:mainlifted} admits a minimizer. 
\end{theo}

\begin{proof}
	As it only shifts the value of the energy $\EE$ by a finite constant, we can always add an integrable function of $x$ to $\Rhs$ and $\phib$, and thus we always assume that the function $a$ for the assumption of being $p$-bounded below or $p$-coercive is identically zero.
	
	We use the direct method of calculus of variations. The problem is not empty, since $\EEu(u) < + \infty$ for at least one $u$, and the lifting identity~\eqref{eq:lifting_identity} holds.
	Let $\{\gam_k\}$ be a minimizing sequence, and let $\dd \tplan_k \coloneqq \dd (\gam_k)_x \dd x$ as a measure on $\bar{\Omega} \times \R^n$. 
	
	\emph{1st case: if $\Rhs$ is $q$-coercive and $\phib$ is $p$-bounded from below}.
	Using~\eqref{eq:trace_lift_int_to_b}, we obtain
	\begin{equation*}
		\EE(\gam) \geq \frac{1}{2} \Dir(\gam) + c \iint_\QQ |y|^q \, \dd \gam_x(y) \dd x - C \iint_\QQ |y|^p \, \dd \gam_x(y) \dd x - C \Dir(\gam)^{p/2}. 
	\end{equation*}
	As $q > p$ and $p/2 < 1$, from $\sup_k \EE(\gam_k) < + \infty$ we see that the minimizing sequence is such that 
	\begin{equation}
		\label{eq:a_priori_F_coercive}
		\sup_k \Dir(\gam_k) < + \infty, \; \qquad \sup_k \iint_\QQ |y|^q \, \dd \tplan_k < + \infty. 
	\end{equation}
	Furthermore, using again~\eqref{eq:trace_lift_int_to_b}, we also have 
	\begin{equation}
		\label{eq:a_priori_F_coercive_boundary}
		\sup_k \iint_{\p \Omega \times \R^n} |y|^{q} \dd (\gamma_k)^b_x(y) \dd \Hau(x) < + \infty.
	\end{equation}
	The second bound in~\eqref{eq:a_priori_F_coercive} is enough to guarantee tightness of $\{ \pi_k \}$, and thus relative compactness in the topology of narrow convergence \cite[Remark 5.1.5]{AmbrosioGigliSavare08}.  
	Thus, up to a subsequence, $\tplan_k$ converges weakly to some $\tplan\in \cP(\QQ)$. Its first marginal is clearly the Lebesgue measure, which implies that there is $\gam\in L^0(\Omega, \cP(\R^n))$ such that $\dd \tplan=\dd \gam_x \dd x$. By Theorem~\ref{t:dir}, $\Dir(\gam)$ is finite and does not exceed $\liminf_{k\to +\infty}\Dir(\gam_k)$. By Theorem~\ref{theo:boundary_values}, $(\gamma_k)^b_x \dd \Hau(x)$ converges weakly to $\gamb_x \dd \Hau(x)$ as $k \to + \infty$. By Corollary~\ref{c:b1} (recall that $F$ is coercive thus bounded from below) and Corollary~\ref{c:b2} (using the bound~\eqref{eq:a_priori_F_coercive_boundary} as $q > p$), 
	\begin{multline*}
		\iint_\QQ \Rhs(x,y) \,\dd\gam_x(y)\,\dd x+ \iint_{\p \QQ} \phib(x,y)\,\dd\gamb_x(y) \,d\Hau(x)\\ \le \liminf_{k\to +\infty}\iint_\QQ \Rhs(x,y) \,\dd(\gam_k)_x(y)\,\dd x+ \iint_{\p \QQ} \phib(x,y)\,\dd(\gamb_k)_x(y) \,\dd\Hau(x). 
	\end{multline*} 
	Summing we get $\EE(\gam) \leq \lim_k \EE(\gam_k)$, hence $\gamma$ is a minimizer of \eqref{e:mainlifted}.
	
	\emph{2nd case: if $\phib$ is $q$-coercive and $\Rhs$ is $p$-bounded from below}.
	We only briefly comment on how to adapt the arguments in this second case. This time we use~\eqref{eq:trace_lift_b_to_int} to have the a priori estimate
	\begin{multline*}
		\EE(\gam) \geq 
		\frac{1}{2} \Dir(\gam) + c \iint_{\p \Omega \times \R^n} |y|^q \dd \gamb_x(y) \dd \Hau(x) \\
		- C \iint_{\p \Omega \times \R^n} |y|^p \dd \gamb_x(y) \dd \Hau(x) - C \Dir(\gam)^{p/2}. 
	\end{multline*}
	Using that $\EE(\gam_k)$ is bounded from above uniformly in $k$, again with~\eqref{eq:trace_lift_b_to_int}, we conclude
	\begin{equation*}
		\sup_k \left\{ \Dir(\gam_k)  + \iint_\QQ |y|^q \, \dd \tplan_k + \iint_{\p \Omega \times \R^n} |y|^{q} \dd (\gamma_k)^b_x(y) \dd \Hau(x) \right\} < + \infty. 
	\end{equation*}
	Once we have these estimates, the rest is almost identical: the sequence $\{\tplan_k\}$ is tight, converges to $\tplan$ such that $\dd \tplan=\dd \gam_x \dd x$, and with Corollary \ref{c:b1} (for the boundary term) and Corollary \ref{c:b2} (for the bulk term), we have $\EE(\gam) \leq \lim_k \EE(\gam_k)$. Thus we conclude that $\gamma$ is a minimizer of \eqref{e:mainlifted}.
\end{proof}

\begin{rmk}
	Under the same assumptions as in Theorem~\ref{theo:existence_lifted}, the direct method of calculus of variations also yields the existence of a solution to the classical problem~\eqref{e:mainnotlifted}.
\end{rmk}

\subsection{Optimality conditions} \label{sec-opt}

Now that existence is guaranteed under standard growth conditions, we turn to the derivation of optimality conditions and explain how they can be interpreted as reasonable measure-valued generalizations of the system \eqref{eq:EL_classical}.
We write $\alpha$ for an index in $1, \ldots, d$, corresponding to the coordinates in $\Omega$. On the other hand, we use the Roman letters $i,j$ to index the coordinates of $\R^n$. We drop the dependence of $\gam$ and $Q$ on $x$ for readability. In particular $Q = (Q^{\alpha i})_{\alpha,i}$ as a collection of scalar-valued measures. 
We also recall the notation $M_{p,\Theta}$ for the $p$-th moment in~\eqref{eq:def_moment}.   

\begin{theo}
	\label{theo:optimality_conditions}
	Assume $\gamma$ is a minimizer in~\eqref{e:mainlifted} and write $Q$ for the tangent flux. Assume $\rhs$ and $\g$ have respectively at most $p$- and $q$-growth in the sense of Definition~\ref{defi:p_bounded_coercive}, $0\le p,q$, and that 
	\begin{equation}
		\label{eq:boundedness_moments}
		M_{p,\Omega}(\gam) < + \infty; \quad M_{q,\p\Omega}(\gamb) < + \infty.
	\end{equation}
	Then, in the sense of distributions, 
	\begin{align}
		\label{eq:optimality_int_acc}
		\sum_{\alpha = 1}^d \partial_{x_\alpha} Q^{\alpha i} + \sum_{\alpha = 1}^d \sum_{j=1}^n \partial_{y_j} \left( \frac{\dd Q^{\alpha i}}{\dd \gam} \frac{\dd Q^{\alpha j}}{\dd \gam} \gam \right)  & +f^i \gam=0, & \forall i =1, \ldots, n \quad \text{in } \Omega \times \R^n.
	\end{align}
	Moreover, the map $x \mapsto Q_x$, seen as a map valued in the dual of $C^1_c(\R^n)^{nd}$, has a divergence which is in $L^1$, and the normal part of the trace is given by 
	\begin{align}
		\label{eq:optimality_boundary}
		\sum_{\alpha=1}^d Q^{\alpha i} \nO^\alpha & = f_b^i \gam^b & \forall i=1, \ldots, n \quad \text{on } \partial \Omega \times \R^n.
	\end{align}
\end{theo}

We also refer the reader to~\eqref{eq:optimality_weak_form} for the weak form of these optimality conditions, which may be easier to process, as well as to Corollary~\ref{cor:optimality_conditions_h} below. 
Note that $\frac{\dd Q}{\dd \gam}$ is a function in $L^2_\gam(\Omega \times \R^n)^{n \times d}$, so that $\frac{\dd Q^{\alpha i}}{\dd \gam} \frac{\dd Q^{\alpha j}}{\dd \gam} \gam$ is a signed measure, whose derivative can be defined in the sense of distributions.

\begin{rmk}
	Write $v$ for the density $\frac{\dd \q}{\dd \gam}$.
	If $\gam$ and $v$ were smooth functions, using the chain rule in~\eqref{eq:optimality_int_acc}, and simplifying with the help of the continuity equation~\eqref{eq:CE}, we see that 
	\begin{align*}
		\sum_{\alpha = 1}^d \partial_{x_\alpha} v^{\alpha i}  + \sum_{\alpha = 1}^d \sum_{j=1}^n v^{\alpha i} \partial_{y_j}  v^{\alpha j}   & + f^i=0,  & \forall i=1, \ldots, n \quad \text{on } \mathrm{supp}(\gam),
	\end{align*}
	which was already derived informally in \cite{B01} in the purely harmonic case $f\equiv 0$, see also \cite{HL19}. Note that, as $Q$ is the tangent flux, $v$ is formally the gradient of a function in the $y$ variable, see Remark~\ref{rem:tangent_gradient}. On the other hand, for~\eqref{eq:optimality_boundary}, if $\gam$, $\gam^b$ and $v$ are continuous functions up to the boundary, then the boundary conditions read
	\begin{align*}
		\sum_{\alpha=1}^d v^{\alpha i} \nO^\alpha & = f_b^i  & \forall i=1, \ldots, n \quad \text{on } \mathrm{supp}(\gam^b).  
	\end{align*}
\end{rmk}

The proof method is to use ``interior'' perturbations, meaning perturbation of the source domain. This is a classical approach when studying harmonic maps \cite[Chapter 8]{Jost17}, and it has also been used by Brenier to study regularity of curves valued in the space of probability measures \cite{brenier1997homogenized}, with notable applications to the study of generalized Euler equations \cite{brenier1999minimal,ambrosiofigalli2009geodesics}. The main difference is that we need to keep track of more than one ``temporal'' dimension, but this is mainly about careful bookkeeping. 

\begin{proof}
	In this proof we use Einstein's convention of summation over repeated indices. We recall that $\alpha$ indexes the coordinates of $x \in \Omega$, while $i,j$ index the coordinates of $y \in \R^n$. 
	
	\emph{1st step: construction of a competitor}. 
	Fix $\psi \in C^1_b(\bar{\Omega} \times \R^n)^n$. For $\epsilon \in \R$ we define 
	\begin{equation*}
		\Phi_\epsilon(x,y) = y + \epsilon \psi(x,y).
	\end{equation*}
	For $|\epsilon|$ small enough, $\Phi_\epsilon(x,\cdot)$ is a bijection from $\R^n$ onto $\R^n$ for all $x \in \Omega$. We define 
	\begin{equation}
		\label{eq:gam_eps_Q_eps}
		\gam^\epsilon_x = \Phi_\epsilon(x,\cdot)_\# \gam_x, \quad \gamma^{b,\epsilon}_x =  \Phi_\epsilon(x,\cdot)_\# \gamb_x, \quad Q^{\epsilon,\alpha i}_x = \Phi_\epsilon(x,\cdot)_\# \left( (\partial_{x_\alpha} \Phi^i_\epsilon)  \gam + (\partial_{y_j} \Phi^i_\epsilon) \q^{\alpha j}_x \right). 
	\end{equation}
	We claim that $(\gam^\epsilon, Q^\epsilon)$ satisfies the continuity equation with boundary conditions $\gamma^{b,\epsilon}$. 
	Indeed, take $\varphi \in C^1_b(\bar{\Omega} \times \R^n)^d$ and define $\varphi_\epsilon(x,y) = \varphi(x, \Phi_\epsilon(x,y))$. As $(\gam,Q)$ satisfies the continuity equation, 
	\begin{equation*}
		\iint_\QQ \Div_x \varphi_\epsilon \, \dd \gam_x \dd x + \iint_\QQ \nabla_y  \varphi_\epsilon : \dd \q_x \, \dd x = \iint_{\p \Omega \times \R^n} \varphi_\epsilon \cdot \nO \, \dd \gamb_x \dd \Hau(x).
	\end{equation*}
	We can compute explicitly:
	\begin{align*}
		\Div_x \varphi_\epsilon(x,y)  &= \Div_x \varphi(x,\Phi_\epsilon(x,y)) + \partial_{x_\alpha} \Phi_\epsilon^i(x,y) \partial_{y_i} \varphi^\alpha(x,\Phi_\epsilon(x,y)), \\
		\partial_{y_i} \varphi_\epsilon^\alpha(x,y) &=  \partial_{y_i} \Phi^j_\epsilon(x,y) \partial_{y_j} \varphi^\alpha(x,\Phi_\epsilon(x,y)).
	\end{align*}
	Note that $\partial_{y_i}$ in the right-hand sides of these equalities denotes the partial derivative with respect to the $i$-th component of the “second” argument of $\varphi^\alpha$, rather than with respect to the actual $y_i$.
	Consequently, we obtain: 
	\begin{multline*}
		\iint_\QQ  \Div_x \varphi \, \dd \gam_x^\epsilon \dd x + \iint_\QQ \partial_{x_\alpha} \Phi_\epsilon^i(x,y) \partial_{y_i} \varphi^\alpha(x,\Phi_\epsilon(x,y)) \dd \gamma_x(y) \dd x \\
		+ \iint_\QQ \partial_{y_i} \Phi^j_\epsilon(x,y) \partial_{y_j} \varphi^\alpha(x,\Phi_\epsilon(x,y)) \dd Q^{\alpha i}_x(y) \dd x = \iint_{\p \Omega \times \R^n} \varphi \cdot \nO \, \dd \gamma^{b,\epsilon}_x \dd \Hau(x),
	\end{multline*}
	where in the first and the last integral we did the change of variables $y \leftrightarrow \Phi_\epsilon(x,y)$, by definition of the push-forward. Then we put the second and third integral together: renaming the indices, 
	\begin{align*}
		\iint_\QQ & \partial_{x_\alpha} \Phi_\epsilon^i(x,y) \partial_{y_i}  \varphi^\alpha(x,\Phi_\epsilon(x,y)) \dd \gamma_x(y) \dd x \\
		& \qquad \qquad + \iint_\QQ \partial_{y_i} \Phi^j_\epsilon(x,y) \partial_{y_j} \varphi^\alpha(x,\Phi_\epsilon(x,y)) \dd Q^{\alpha i}_x(y) \dd x \\
		& = \iint_\QQ \partial_{y_i} \varphi^\alpha(x,\Phi_\epsilon(x,y)) \Bigg(  \partial_{x_\alpha} \Phi_\epsilon^i(x,y)  \dd \gamma_x(y)  + \partial_{y_j} \Phi^i_\epsilon(x,y) \dd Q^{\alpha j}_x(y) \Bigg) \dd x \\
		& = \iint_\QQ \nabla_y \varphi : \dd \q^\epsilon_x \, \dd x
	\end{align*}
	with $\q^\epsilon$ defined as in~\eqref{eq:gam_eps_Q_eps}. Thus, 
	\begin{equation*}
		\iint_\QQ \Div_x \varphi \, \dd \gam^\epsilon_x \dd x + \iint_\QQ \nabla_y \varphi : \dd \q^\epsilon_x \, \dd x = \iint_{\p \Omega \times \R^n} \varphi \cdot \nO \, \dd \gamma^{b,\epsilon}_x \dd \Hau(x).
	\end{equation*}
	As this is valid for any smooth $\varphi$, this gives us the claim. 
	
	\emph{2nd step: variation of the energy}. We now evaluate the energy $\EE(\gam^\epsilon)$. We start with the Dirichlet energy: doing again for a fixed $x$ the change of variables $y \leftrightarrow \Phi_\epsilon(x,y)$, we find easily 
	\begin{align*}
		\Dir(\gam^\epsilon)  \leq \iint_\QQ \left\|\frac{\dd \q^\epsilon_x}{\dd \gam^\epsilon_x} \right\|^2 \dd \gam^\epsilon_x \, \dd x & = \sum_{\alpha, i} \iint_\QQ \left(\frac{\dd ( (\partial_{x_\alpha} \Phi^i_\epsilon)  \gam_x + (\partial_{y_j} \Phi^i_\epsilon) \q^{\alpha j}_x)}{\dd \gam_x} \right)^2 \dd \gam_x \, \dd x \\
		& = \sum_{\alpha, i} \iint_\QQ \left(  \partial_{x_\alpha} \Phi^i_\epsilon + \partial_{y_j} \Phi^i_\epsilon \frac{ \dd \q^{\alpha j}_x}{\dd \gam_x} \right)^2 \dd \gam_x \, \dd x.
	\end{align*}
	Recalling $\Phi_\epsilon(x,y) = y + \epsilon \psi(x,y)$, we have
	\begin{equation*}
		\partial_{x_\alpha} \Phi^i_\epsilon(x,y) = \epsilon \partial_{x_\alpha} \psi^i(x,y) , \qquad
		\partial_{y_j} \Phi^i_\epsilon(x,y) = \delta_{ij} + \epsilon  \partial_{y_j} \psi^i(x,y).
	\end{equation*}
	Thus at first order in $\epsilon$ we obtain 
	\begin{equation*}
		\Dir(\gam^\epsilon) \leq \Dir(\gam) + 2 \epsilon \iint_\QQ \left( \partial_{x_\alpha} \psi^i(x,y) +   \partial_{y_j} \psi^i(x,y) \frac{\dd \q^{\alpha j}}{\dd \gam_x}(y)  \right)  \dd Q^{\alpha i}_x \, \dd x + \mathcal{O}(\epsilon^2).
	\end{equation*}
	On the other hand, for the zero-th order part of the energy, from $\nabla \Rhs = -f$, and as $f$ has at most $p$-growth while we have the moment bound~\eqref{eq:boundedness_moments}, we can differentiate 
	under the integral sign
	and obtain
	\begin{align*}
		\iint_\QQ   \Rhs \,\dd\gam^\epsilon_x\,\dd x   & = \iint_\QQ \Rhs(x,\Phi_\epsilon(x,y)) \,\dd\gam_x(y)\,\dd x \\
		& = \iint_\QQ \Rhs\,\dd\gam_x\,\dd x - \epsilon \iint_\QQ \psi(x) \cdot f  \,\dd\gam_x\,\dd x + o(\epsilon).
	\end{align*}
	A similar computation can be done for $\phib$.
	Writing that $\EE(\gam^\epsilon) \geq \EE(\gam)$ by optimality, given the definition of $\EE$ in~\eqref{e:mainliftedenergy}, and as well as the two estimates above, collecting the terms of order one in $\epsilon$ we have
	\begin{multline*}
		\iint_\QQ \left( \partial_{x_\alpha} \psi^i(x,y) +   \partial_{y_j} \psi^i(x,y) \frac{\dd \q^{\alpha j}}{\dd \gam_x}(y) \right)  \dd Q^{\alpha i}_x(y) \, \dd x \\
		- \iint_\QQ \psi^i(x,y) f^i(x,y) \,d\gam_x(y)\,\dd x - \iint_{\p \QQ} \psi^i(x,y)  f_b^{i}(x,y)\,\dd\gamb_x(y) \,\dd\Hau(x) \geq 0. 
	\end{multline*}
	
	\emph{3rd step: reading the optimality conditions}.
	Changing $\psi$ in $-\psi$, we actually have an equality in the inequality above. We rewrite it as: for any $\psi \in C^1_b(\bar{\Omega} \times \R^n)^n$
	\begin{multline}
		\iint_\QQ  \partial_{x_\alpha} \psi^i(x,y)   \dd Q^{\alpha i}_x(y) \, \dd x + \iint_\QQ \partial_{y_j} \psi^i(x,y)   \left( \frac{\dd \q^{\alpha j}}{\dd \gam_x}(y) \frac{\dd Q^{\alpha i}_x}{\dd \gam_x}(y)   \right) \dd \gam_x(y)  \dd x \\
		=  \iint_\QQ  \psi^i(x,y)  f^i(x,y)  \dd \gam_x(y)   \dd x + \iint_{\p \QQ} \psi^i(x,y)  f^i_b(x,y)\,\dd\gamb_x(y) \,\dd\Hau(x) 
		\label{eq:optimality_weak_form}
	\end{multline}
	where we recall that we sum over repeated indices. 
	Taking $\psi$ which is compactly supported in the interior of $\Omega$, the last term vanishes and we recognize in~\eqref{eq:optimality_weak_form} exactly the weak form of~\eqref{eq:optimality_int_acc}.  Then we fix $\chi \in C^1_c(\R^n)^n$ and choose $\psi(x,y) = \eta(x) \chi(y)$ for $\eta \in C^1(\bar{\Omega})$. We see from~\eqref{eq:optimality_weak_form} that, if we define the function $g : x \mapsto \int_{\R^n} \chi^i(y) \dd Q^{\alpha i}_x(y) \in \R^d$, then $g$ has a divergence in $L^1(\Omega)$ and has normal trace
	\begin{equation*}
		g^b(x) \cdot \nO(x) = \int_{\R^n}  \chi(y) f_b(x,y)\,\dd\gamb_x(y)
	\end{equation*}
	for $\Hau$-a.e. $x \in \p \Omega$. As this is valid for any $\chi$, this precisely gives us~\eqref{eq:optimality_boundary}.
\end{proof}

To further understand our optimality conditions, we link it to the concept of renormalized solutions, similar to the theory of DiPerna-Lions for transport and Boltzmann equations \cite{DiPerna-Lions89,DiPernaLions89A}. In the present case,
assume $u$ is a smooth solution to~\eqref{eq:EL_classical}, and write $W = \sum_{\alpha=1}^d \p_{x_\alpha} u \otimes \p_{x_\alpha} u$, as a matrix-valued field. If $h : \R^n \to \R$ is a smooth function, then the chain rule shows that the function $g : x \mapsto h(u(x))$ satisfies the elliptic system
\begin{equation*}
	\begin{cases}
		-\Delta g(x) + \nabla^2 h(u(x)) : W(x) = \nabla h(u(x)) \cdot f(x,u(x)) & \text{in } \Omega, \\
		\displaystyle{\frac{\partial g}{\partial n_\Omega}(x) = \nabla h(u(x)) \cdot f_b(x,u(x))}  & \text{on } \p \Omega,
	\end{cases}
\end{equation*}
We prove that our measure-valued optimality conditions imply a lift of this ``renormalized'' formulation.

\begin{cor}
	\label{cor:optimality_conditions_h}
	Assume $\gamma$ is a minimizer in~\eqref{e:mainlifted}, write $Q$ for the tangent flux, and assume the same integrability conditions as in Theorem~\ref{theo:optimality_conditions}. We define the matrix-valued field $W \in L^1(\Omega, \cM(\Omega)^{n \times n})$
	\begin{equation*}
		W^{ij}_x = \sum_{\alpha=1}^d \frac{\dd Q^{\alpha i}_x}{\dd \gam_x} \frac{\dd Q^{\alpha j}_x}{\dd \gam_x} \gam_x
	\end{equation*}
	Consider $h : \R^n \to \R$ smooth with compact support and define $$g(x) = \int_{\R^n} h(y) \, \dd \gam_x(y).$$ Then $g$ satisfies in a weak sense 
	\begin{equation} \label{e:renorm}
		\begin{cases}
			\displaystyle{-\Delta g(x) + \int_{\R^n} \nabla^2 h(y) : \dd W_x(y)  = \int_{\R^n} \nabla h(y) \cdot f(x,y)  \, \dd \gam_x(y)} & \text{in } \Omega, \\[3ex]
			\displaystyle{\frac{\partial g}{\partial n_\Omega}(x) = \int_{\R^n} \nabla h(y) \cdot f_b(x,y) \dd \gamb_x(y)} & \text{on } \p \Omega. \\
		\end{cases}
	\end{equation}
\end{cor}

\begin{proof}
	From Lemma~\ref{lem:weak_der_int}, $g$ is in ($W^{1,1}\cap L^\infty)(\Omega)$ with $\nabla g(x) = \int_{\R^n} \nabla h(y) \dd Q_x(y)$. 
	Thus, if $\eta \in C^\infty(\bar{\Omega})$, using~\eqref{eq:optimality_weak_form} with $\psi(x,y) = \eta(x) \nabla h(y)$ and Fubini's theorem 
	\begin{align*}
		\int_\Omega \nabla \eta(x) \cdot \nabla g(x)  \, \dd x   = & - \int_\Omega \eta(x) \left( \int_{\R^n} \nabla^2 h(y) :  W_x(y) \dd \gam_x(y) \right)  \dd x \\
		& + \int_\Omega \eta(x) \left( \int_{\R^n} \nabla h(y) \cdot f(x,y) \dd \gam_x(y) \right)  \dd x \\
		& + \int_{\p\Omega} \eta(x) \left( \int_{\R^n} \nabla h(y) \cdot f_b(x,y) \dd \gamb_x(y) \right) \,\dd\Hau(x), 
	\end{align*}
	which is exactly the weak form of \eqref{e:renorm}. 
\end{proof}

\begin{rmk} In the pioneering works of DiPerna and Lions, the test functions had a scalar argument (since they were not studying systems), and this scalar structure seems to have been important because of certain commutator issues. Nevertheless, some authors have studied systems using a similar approach, cf. \cite{JF16}.
\end{rmk}

\subsection{Comparison with DiPerna's measure-valued solutions} \label{rem:dip} This section is a digression that serves to highlight the novelty of our approach relative to the classical framework of DiPerna \cite{DiPerna85}. Let
\[
\nu = (\nu_x)_{x \in \Omega},
\qquad
\nu_x \in \mathcal{P}(\mathbb{R}^n),
\]
be a measurable family of probability measures on $\mathbb{R}^n$. 
Mimicking\footnote{DiPerna did not introduce his notion for elliptic problems, but rather for first-order hyperbolic equations. This idea has subsequently been developed primarily within the hyperbolic PDE and fluid dynamics communities, cf.~\cite{DiPernaM87,MNRR96,BLS11,FGW16,SV17}. Here, we are discussing a straightforward analogue of DiPerna’s notion in the elliptic setting.} the celebrated  definition from \cite[p. 235]{DiPerna85}, we say that $\nu$ is a \emph{DiPerna measure-valued solution} to
\begin{equation}
	-\Delta u = f(x,u) \label{eeq}
\end{equation} (for simplicity, here we temporarily ignore the boundary conditions)
if for every test function $\varphi \in C_c^\infty(\Omega)$,
\[
\int_{\Omega}
\left(
\int_{\mathbb{R}^n} y \, \Delta \varphi(x) \, \dd \nu_x(y)
\right)
\dd x
+
\int_{\Omega}
\left(
\int_{\mathbb{R}^n} f(x,y) \, \dd \nu_x(y)
\right)
\varphi(x) \, \dd x
=
0.
\]

Equivalently, let
\[
\bar{u}(x) \coloneqq \int_{\mathbb{R}^n} y \, \dd \nu_x(y)
\]
be the barycenter of $\nu_x$, and define
\[
\overline{f(\nu_x)} \coloneqq \int_{\mathbb{R}^n} f(x,y) \, \dd\nu_x(y).
\]
Then DiPerna's measure-valued formulation can be written as
\begin{equation} \label{e:dpstr}
	-\Delta \bar{u} = \overline{f(\nu_x)} 
\end{equation}
in the sense of distributions on $\Omega$.

A natural question is what is the connection between our measure-valued solutions to~\eqref{eeq}, in the sense of Section~\ref{clas-lif}, and DiPerna solutions. 

Consider first the harmonic case $f\equiv 0$. Then, the Ishihara property \cite{HL19,LMTV24} implies, at least formally, that the first moment of our solution $\gamma_x$ is both subharmonic and superharmonic in $x$ (because the first moment functional is geodesically affine on the Wasserstein space). Consequently,
\[
-\Delta \bar{u}=0,
\]
and hence our solution is necessarily a DiPerna solution. The converse is, of course, not true, since any measure-valued map $\nu_x$ with harmonic barycenter is harmonic in the sense of DiPerna, but need not minimize the corresponding energy $\EE$. 


For a general $f$, the DiPerna-style definition above still relies only on the barycenter of a measure-valued map to compute the Laplacian. The fact that the solution is measure-valued significantly affects only the zeroth-order, possibly nonlinear, term $f$. By contrast, our approach, based on the Wasserstein lift, genuinely computes a Laplacian at the level of measure-valued maps. This appears to be a fundamental advantage. As we have seen in Section~\ref{sec-opt}, our definition captures substantially more information in a much more refined manner. More specifically, letting $h(y)=y_i,\  i=1,\dots,n$, in~\eqref{e:renorm}, we formally\footnote{Rigorous derivation of~\eqref{e:renorm} for linear $h$, hence non-compactly supported, lies beyond the scope of this article.} derive 
$$
-\Delta \left[\int_{\R^n} y \, \dd \gam_x(y)\right] = \int_{\R^n} f(x,y)  \, \dd \gam_x(y),
$$
which is exactly \eqref{e:dpstr}. This suggests that a solution $\gamma$ of our lifted problem~\eqref{e:mainlifted} is, in particular, a DiPerna measure-valued solution in the sense of Section~\ref{rem:dip}. The converse is not expected to hold, as discussed above.  Notably, our definition retains the refined nonlinear term involving the field $W$ in~\eqref{e:renorm}, which is not captured by DiPerna's approach. 

\subsection{Convex case} \label{s:consis}

We move to proving the ``consistency'' of our lifted problem with respect to the classical one: if $\Rhs$ and $\phib$ are convex as functions of their second argument, then nothing is gained from going from the unlifted to the lifted problem, in the sense that, provided a classical solution exists, its Dirac lift is a deterministic solution of the lifted problem.

\begin{theo}
	Let $\Rhs$ and $\phib$ be a pair of Carath\'eodory functions. Assume that $\Rhs(x,\cdot)$ and $\phib(x,\cdot)$ are convex for a.e. $x$ and that there exists at least one constant function $u$ such that $\EEu(u) < + \infty$.
	Then we have
	\begin{equation}
		\label{eq:equality_inf}
		\inf_{\gam\in L^0(\Omega, \cP(\R^n))} \EE(\gam) = \inf_{u\in L^0(\Omega,\R^n)} \EEu(u)
	\end{equation}
	(being potentially both infima $- \infty$) and, if there exists a solution $u^*$ to~\eqref{e:mainnotlifted}, then $\gam^*:x \mapsto \delta_{u^*(x)}$ is a solution to~\eqref{e:mainlifted}.
\end{theo}

As every classical competitor $u$ can be lifted into the deterministic map $\gam_x = \delta_{u(x)}$ with $\EE(\gam) = \EEu(u)$, we always have 
\begin{equation}
	\label{eq:ineq_obvious_lift}
	\inf_\gam \EE(\gam) \leq \inf_u \EEu(u),
\end{equation}
so we only need to prove the reverse inequality. The overall idea of the proof is that, from a measure-valued solution $\gam$, we can average it and define $u(x) = \int_{\R^n} y \dd \gam_x(y)$. That defines a classical solution with lower energy thanks to Jensen's inequality. However, if we do not have a finite first moment in the $y$ variable, then we cannot define directly such a $u$. We need first to project the solution in a bounded set in the $y$ variable in this case. 

\begin{proof}
	Let us take $\gam \in L^0(\Omega,\cP(\R^n))$. Without loss of generality, we only need to handle the case $\EE(\gam) \in \R$, which means in particular
	\begin{equation}
		\label{eq:bound_a_priori_competitor}
		\Dir(\gam) < + \infty; \; \iint_\QQ |\Rhs| \, \dd \gam_x \dd x < + \infty; \; \iint_{\p \Omega \times \R^n} |\phib| \dd \gamb_{x} \dd \Hau(x) < + \infty. 
	\end{equation}
	
	\emph{1st step: reducing to a bounded set in the $y$ variable.} We first claim that for any $\epsilon>0$ there exist a ball $D\subset \R^n$ and $\tilde \gam \in L^0(\Omega,  \cP(D))$ such that  
	\begin{equation} \label{e:epsil1} 
		\EE( \gam)+3\epsilon \geq \EE(\tilde \gam). 
	\end{equation}
	Indeed, let $y_0 \in \R^n$ be such that the constant function equal to $y_0$ has finite $\EEu$ energy. Equivalently, $\Rhs(\cdot,y_0)$ and $\phib(\cdot,y_0)$ are integrable functions.
	We take $D=B(y_0,R)\subset \R^n$ a ball independent of $x$. We define the probability measure $\tilde \gam_x$ as the pushforward of $\gam_x$ by a non-expansive retraction $r_D:\R^n\to D$, 
	$$r_D(y) \coloneqq y_0+ h(|y-y_0|) \frac {y-y_0}{|y-y_0|},$$ 
	where $h(s)$ is a smooth strictly increasing non-expansive function of $s\geq 0$ such that $h(s)=s$ for $s\leq R/2$ and $h(s)<\min (R,s)$ for $s>R/2$. It follows from the argument of the proof of \cite[Lemma A.4]{HL23} that 
	\begin{equation} \label{e:direps} 
		\Dir( \gam)+\epsilon  \geq \Dir(\tilde\gamma)
	\end{equation} 
	provided $R$ is sufficiently large.  We now claim that, for sufficiently large $R$, 
	\begin{equation} \label{e:epsest1} 
		\iint_\QQD \Rhs(x,y) \,\dd \tilde\gam_x(y)\,\dd x \leq \iint_\QQ \Rhs(x,y) \,\dd\gam_x(y)\,\dd x+\epsilon.  
	\end{equation}
	To see this, leveraging the convexity of $\Rhs$ w.r.t.\ $y$, we get for every $\alpha \in [0,1]$ and every $y \in \R^n$,
	\begin{equation*} 
		\Rhs(x,y_0+\alpha (y-y_0))-\Rhs(x,y) \leq (\alpha-1) \Rhs(x,y)+(1-\alpha) \Rhs(x,y_0)
		\leq  |\Rhs(x,y)|+ |\Rhs(x,y_0)|.
	\end{equation*}
	Thus we can  estimate 
	\begin{align*} 
		\iint_\QQD \Rhs(x,y) \,\dd \tilde \gam_x(y)\,\dd x & -  \iint_\QQ \Rhs(x,y) \,\dd\gam_x(y)\,\dd x \\ 
		& =  \iint_\QQ \left[\Rhs(x,r_D(y)) - \Rhs(x,y)\right] \,\dd \gam_x(y)\,\dd x \\
		& = \iint_{\Omega \times (\R^n \backslash B_{R/2}(y_0))} \left[\Rhs(x,r_D(y)) - \Rhs(x,y)\right] \,\dd \gam_x(y)\,\dd x\\ 
		& \leq  \iint_{\Omega \times (\R^n \backslash B_{R/2}(y_0))} ( |\Rhs(x,y)|+ |\Rhs(x,y_0)| ) \,\dd\gam_x(y)\,\dd x. 
	\end{align*}
	Since the last integrand is in $L^1(\QQ)$ w.r.t. the probability  measure $\dd \gam_x \dd x$ thanks to~\eqref{eq:bound_a_priori_competitor} (and thus it is uniformly integrable), for $R$ sufficiently large the last term is arbitrarily small. We deduce~\eqref{e:epsest1}. Next, the exact same reasoning yields that for sufficiently large $R$, 
	\begin{equation} \label{e:epsest2} 
		\iint_{\p \QQD} \phib(x,y) \,\dd\tilde{\gamma}^b_x(y)\,\dd \Hau(x) \leq \iint_{\p \QQ} \phib(x,y) \,\dd\gamb_x(y)\,\dd \Hau(x)+\epsilon. 
	\end{equation} 
	Eventually, combining~\eqref{e:direps}, \eqref{e:epsest1} and~\eqref{e:epsest2}, we obtain the estimate~\eqref{e:epsil1}.
	
	\emph{2nd step: averaging in $y$ and concluding by Jensen.} Now that $\tilde{\gam}$ is supported on a compact set in the $y$ variable, we can use a standard procedure to average it into a classical function. Specifically, we claim that, defining
	\begin{equation*}
		\tilde{u}(x) = \int_D y \, \dd \tilde \gam_x(y), 
	\end{equation*} 
	(note that $\tilde{u}$ is well defined as $D$ is compact) then we have 
	\begin{equation}
		\label{e:epsil2}
		\EEu(\tilde{u}) \leq \EE(\tilde{\gamma}).
	\end{equation}
	Indeed, let us call $\tilde{\q}$ the tangent flux. From Lemma~\ref{lem:weak_der_int} and Lemma~\ref{lem:weak_der_int_b} (using for $\varphi(x,y)$ a function independent of $x$ which coincides with $y$ on $D$), $\tilde{u} \in (W^{1,1}\cap L^\infty)(\Omega)$ with $\nabla \tilde{u}(x) = \tilde{\q}_x(D)$ and boundary values $\tilde{u}(x) = \int_D y \, \dd \tilde{\gamma}^b_x(y)$ for $x \in \p \Omega$. 
	Jensen's inequality gives
	\begin{equation*}
		\int_\Omega \|\nabla \tilde{u}(x)\|^2 \, \dd x = \int_\Omega \left\| \int_D \frac{\dd \tilde{\q}_x}{\dd \tilde{\gam}_x} \dd \gam_x(y) \right\|^2 \dd x \leq \Dir(\tilde{\gam}).
	\end{equation*}
	The convexity of $\Rhs$ and $\phib$, again with Jensen's inequality, yields
	\begin{align*}
		\int_\Omega \Rhs(x,\tilde{u}(x)) \, \dd x & \leq \iint_\QQD \Rhs(x,y) \,\dd \tilde\gam_x(y)\,\dd x \\ 
		\int_{\p \Omega} \phib(x,\tilde{u}^b(x)) \, \dd \Hau(x) & \leq \iint_{\p \QQD} \phib(x,y) \,\dd\tilde{\gamma}^b_x(y)\,\dd \Hau(x).
	\end{align*}
	Summing the three last inequalities yields~\eqref{e:epsil2}. Together with~\eqref{e:epsil1} we have
	\begin{equation*}
		\EEu(\tilde{u}) \leq \EE(\tilde{\gam}) \leq \EE(\gam) + 3 \epsilon,
	\end{equation*}
	which combined with~\eqref{eq:ineq_obvious_lift} and the arbitrariness of $\epsilon$ gives~\eqref{eq:equality_inf}.
	
	Once we have~\eqref{eq:equality_inf}, it is clear that if $u^*$ is a minimizer of $\EEu$ then $\gam^* : x \mapsto \delta_{u^*(x)}$ is a minimizer of $\EE$. 
\end{proof}

\subsection{One-dimensional case} \label{s:1d}

We consider the case where $\Omega$ is a segment of $\R$, that we call $I$. Up to renormalization, we have $I = [0,1]$ the unit interval. In particular $\phib$ is the collection of two functions $\phib(0,\cdot)$ and $\phib(1,\cdot)$. Let us recall the superposition principle for curves of probability measures. In the sequel, $e_x : C(I,\R^n) \to \R^n$ is the evaluation map of a continuous function at time $x \in I$.

\begin{theo}
	\label{theo:superposition}
	With $I = [0,1]$ take $\gam \in L^0(I,  \cP(\R^n))$ such that $\Dir(\gam) < + \infty$. Then there exists $P \in \cP(H^1(I,\R^n))$ such that $\gam_x = {e_x}_\# P$ for all $x \in [0,1]$, and 
	\begin{equation*}
		\Dir(\gam) = \int_{H^1(I,\R^n)} \left( \int_0^1 |\dot{u}(x)|^2 \dd x \right) \dd P(u).
	\end{equation*}
\end{theo}

\begin{proof}
	The continuity equation~\eqref{eq:CE} reduces to the classical continuity equation. From \cite[Theorem~8]{Lisini07} we deduce that $x \mapsto \gam_x$ belongs to $\mathrm{AC}_2(I, \cP(\R^n))$. Here $(\cP(\R^n),W_2)$ is an extended metric space, but we can still apply the superposition principle, see \cite[Corollary~1]{Lisini07} and also \cite{Lisini16}.
\end{proof}

\begin{cor} \label{cor-det}
	With $\Omega = [0,1]$, we have
	\begin{equation*}
		\inf_\gam \EE(\gam) = \inf_u \EEu(u); 
	\end{equation*}
	(being potentially both infima $- \infty$) and if there exists a solution $u^*$ to~\eqref{e:mainnotlifted} then $\gam^*:x \mapsto \delta_{u^*(x)}$ is a solution to~\eqref{e:mainlifted}.
\end{cor}

\begin{proof}
	Due to~\eqref{eq:ineq_obvious_lift}, we always have $\inf_\gam \EE(\gam) \leq \inf_u \EEu(u)$, so we want to prove the reverse inequality. Let us take $\gam \in L^0(\Omega,\cP(\R^n))$. Without loss of generality, we only need to handle the case $\EE(\gam) \in \R$. In particular, $\Dir(\gam) < + \infty$, so we can apply Theorem~\ref{theo:superposition}: we deduce that there exists $P \in \cP(H^1([0,1],\R^n))$ such that 
	\begin{align*}
		\EE(\gam) & = \frac{1}{2} \int_{H^1(I,\R^n)} \left( \int_0^1 |\dot{u}(x)|^2 \dd x \right) \dd P(u)
		+ \iint_{I \times \R^n} \Rhs(x,y) \,\dd\gam_x(y)\,\dd x \\
		& \qquad + \int_{\R^n} \phib(0,y)\,\dd\gamb_0(y) + \int_{\R^n} \phib(1,y)\,\dd\gamb_1(y)  \\
		& = \int_{H^1(I,\R^n)} \EEu(u) \, \dd P(u).
	\end{align*}
	The conclusion $\EE(\gam) \geq \inf \EEu$ follows, thus the result.
\end{proof}
\begin{rmk} At first glance, if we drop the RHS term involving $\Rhs$ (i.e., if $\Rhs\equiv 0$), the one-dimensional problem considered in this section appears to fit within the primal setting of entropy-transport problems studied in \cite{LMS18}. However, this is not the case, because our boundary terms (say, at $t=0$) are of the form $\int_{\R^n} \phib(0,y)\, \dd\gamb_0.$ If we try to write this as a relative entropy of $\gamb_0$ with respect to some measure (cf.\ their formula~(2.35)), the corresponding entropy function would have to be linear, in which case $\phib(0,y)$ would necessarily be constant. In our setting, however, $\phib(0,y)$ depends on $y$. Moreover, even in the case when $\phib(0,y)$ is constant, their basic assumptions exclude linear relative entropies. The fact that they excluded the linear case is not surprising, since Corollary~\ref{cor-det} indicates that~\eqref{e:mainlifted} is not genuinely measure-valued in dimension $1$. \end{rmk} 

\begin{rmk} \label{brenier} A related minimization problem is obtained by taking $\Omega=[0,1]$, $\Rhs\equiv 0$, and $\phib(1,\cdot)\equiv 0$, together with the additional constraint that $\gamb_1$ is fixed, which means that a measure-valued Dirichlet boundary condition is imposed at $1$, while an inhomogeneous Neumann boundary condition is still prescribed at $0$. As explained in \cite[Appendix B]{V25}, this is the so-called ballistic optimal transport, see also \cite{V22,MST26,V26,AK25,AP25}. This problem belongs to a broader class of variational problems whose underlying minimax framework was originally proposed by Brenier in \cite{CMP18}. Such problems arise naturally from the dual formulation of various PDEs, cf. \cite{V22,A23}. Note that $\phib(0,\cdot)$ can be viewed a posteriori as the Kantorovich potential for the Monge--Kantorovich problem with marginals $\gamb_0$ and $\gamb_1$, once the originally unknown $\gamb_0$ has been determined. \end{rmk}



\subsection{Genuinely measure-valued case} \label{s:measurev}

Finally, we provide two examples that fall within the framework of our existence theory (Theorem~\ref{theo:existence_lifted}) in which both the non-lifted problem~\eqref{e:mainnotlifted} and the lifted problem~\eqref{e:mainlifted} admit solutions, but their optimal values differ. In view of~\eqref{eq:ineq_obvious_lift}, the optimal values of the corresponding lifted problems cannot be attained by deterministic maps, showing that the measure-valued solutions cannot be deterministic. The construction is inspired by \cite[Proposition~5.5]{HL19} and thus seems to be related to the failure \cite{HL19} of a superposition principle for Wasserstein-valued harmonic functions with multidimensional sources.

\begin{example}\label{ex:value}
	Set $n=d=2$, $\Omega = B_1(0) \subset \R^2$.
	We will identify vectors from $\R^2$ with complex numbers.
	Fix $0<r_0<\frac12$, and select
	$\eta\in C^\infty(\overline{\Omega}\times\mathbb{R}^2)$
	with the following properties:
	\begin{enumerate}
		\item $0\leq \eta\leq 1$;
		\item $\eta(x,y)=0$ for $|x|\leq r_0$;
		\item $\eta(x,y)=0$ whenever $x=re^{i\theta}\neq 0$ and
		$y=\pm e^{i\theta/2}$;
		\item if $r\geq 2r_0$, $x=re^{i\theta}$ and $\eta(x,y)=0$, then
		$y=\pm e^{i\theta/2}$.
	\end{enumerate}
	Take also a non-negative function $\zeta \in C^\infty(\R^2)$ such that 
	\begin{enumerate}
		\item $\zeta(y) = |y|^2$ outside of $B_2(0)$;
		\item $\zeta(y) \equiv 0$ on $B_1(0)$.
	\end{enumerate}
	For all $\lambda > 0$ set
	\[
	F_\lambda(x, y) = \lambda (\eta(x,y) + \zeta(y)).
	\]
	We consider the problems~\eqref{e:mainnotlifted}, \eqref{e:mainlifted} with $\Rhs = F_\lambda$ and the boundary term $\phib \equiv  0$. Denote the corresponding energy functionals by $\EEu_\lambda$ and $\EE_\lambda$, respectively.
	Note that $F_\lambda$ is continuous, nonnegative, and 2-coercive, thus both problems admit a solution.
	
	Now, fix a measure-valued map $\bar{\gam}: \Omega \to \cP(\R^2)$ such that
	\[
	\bar{\gam}_x = \frac{1}{2} (\delta_{f(r) e^{i \theta/2}} + \delta_{-f(r) e^{i \theta/2}}), \quad x = r e^{i \theta},
	\]
	where $f$ is a smooth monotone function such that $f(r) = 0$ if $r \le \frac{r_0}{2}$ and $f(r) = 1$ if $r \ge r_0$.
	Clearly, $\gam$ is Lipschitz continuous in the Wasserstein metric, thus $\Dir(\bar{\gam}) < \infty$.
	Moreover, it satisfies by construction 
	\[
	\iint_{\Omega \times \R^2} F_\lambda(x, y) \,d\bar{\gam}_x(y)\,\dd x = 0, \quad \forall \lambda > 0.
	\]
	Therefore, for any $\lambda > 0$
	\[
	\min_{\gam\in L^0(\Omega, \cP(\R^2))} \EE_\lambda(\gam) \le \EE_\lambda(\bar{\gam}) = \frac{1}{2} \Dir(\bar{\gam}) < \infty.
	\]
	
	Suppose that our key claim does not hold and for any $\lambda > 0$ we have \[
	\min_{u\in L^0(\Omega,\R^2)} \EEu_\lambda(u) = \min_{\gam\in L^0(\Omega, \cP(\R^2))} \EE_\lambda(\gam).
	\] Then for any $\lambda > 0$ 
	\[
	\min_{u\in L^0(\Omega,\R^2)} \EEu_\lambda(u) \le \frac{1}{2} \Dir(\bar{\gam}).
	\]
	Denote by $u_\lambda$ the solution to the problem on the left-hand side.
	From the coercivity and the nonnegativity of $F_\lambda$ it follows that the norms $\|u_\lambda\|_{H^1}$ are uniformly bounded for $\lambda \ge 1$.
	Extract a subsequence $u_k = u_{\lambda_k}$ weakly converging to $u^* \in H^1$, with $\lambda_k \to \infty$.
	Since $\lambda \mapsto F_\lambda$ is pointwise monotone, we conclude that
	\[
	\EEu_{\lambda_k}(u^*) \le \liminf_{m \to \infty} \EEu_{\lambda_k}(u_m) \le \liminf_{m \to \infty} \EEu_{\lambda_m}(u_m) \le \frac{1}{2} \Dir(\bar{\gam})
	\]
	for all $k$. Thus,
	\[
	\int_{\Omega} \eta(x, u^*(x)) \,\dd x + \int_{\Omega} \zeta(u^*(x)) \,\dd x = \frac{1}{\lambda_k} \int_{\Omega} F_{\lambda_k}(x, u^*(x)) \,\dd x \to 0.
	\]
	This yields that
	\[
	\int_{\Omega} \eta(x, u^*(x)) \,\dd x = \int_0^1 \int_0^{2 \pi} \eta(r e^{i \theta}, u^*(r e^{i \theta}))\, r\,\dd \theta \, \dd r = 0,
	\]
	and hence $u^*(r e^{i \theta}) = \pm e^{i \theta / 2}$ for a.e.\ $(r, \theta) \in [2 r_0, 1] \times [0, 2 \pi]$.
	However, the boundary trace of $u^*\in H^1(B_r(0))$, $2r_0<r<1$, belongs to $H^{1/2}(\partial B_r(0))$, see~\cite[Ch. 16 and 19]{T2007}, and thus cannot have a jump discontinuity, see~\cite[Ch. 33]{T2007}.
	
	Consequently, for large enough $\lambda$, 
	\[
	\min_{u\in L^0(\Omega,\R^2)} \EEu_\lambda(u) > \frac{1}{2} \Dir(\bar{\gam}) \ge \min_{\gam\in L^0(\Omega, \cP(\R^2))} \EE_\lambda(\gam) .
	\]
\end{example}

In the same way, one constructs a problem with $F \equiv 0$ and non-trivial $\phib$.

\begin{example}\label{ex:value_boundary}
	Consider the same setting as in Example~\ref{ex:value}.
	Let $F \equiv 0$ and
	\[
	\phib^\lambda \equiv F_\lambda \text{~on~} \partial \Omega \times \R^2.
	\]
	Again, both the lifted and the non-lifted problems admit a solution,
	and the measure-valued map $\bar{\gamma}$ has uniformly bounded energy for all $\lambda > 0$.
	Assuming that the values of the lifted and non-lifted problems are the same, we infer that
	\begin{equation} \label{sob_norm}
		\|\nabla u_\lambda\|_{L^2(\Omega)} + \|u_\lambda\|_{L^2(\partial \Omega)}
	\end{equation}
	is uniformly bounded for all $\lambda \ge 1$, where $u_\lambda$ is a solution to the non-lifted problem.
	Since~\eqref{sob_norm} is equivalent to the standard norm in $H^1(\Omega)$, cf.\ \cite{B_cons}, for large $\lambda > 0$ we conclude by contradiction that $\min_{u} \EEu_\lambda(u) > \min_{\gam} \EE_\lambda(\gam)$ in exactly the same way as in Example~\ref{ex:value}.
\end{example}




\appendix
\section{Lower semi-continuity for lifted Carathéodory functionals}
\label{sec:appendix_caratheodory}

In this section $(\Theta,\sigma)$ is a Polish space endowed with a finite measure. We always use the results for $(\Theta,\sigma) = (\Omega,\dd x)$ or $(\Theta,\sigma) = (\p \Omega, \Hau)$.
We recall that a function $\phi : \Theta \times \R^n \to \R$ is called Carath\'eodory if $\phi(\cdot,y)$ is measurable for all $y \in \R^n$ and $\phi(\theta,\cdot)$ is continuous for $\sigma$-a.e. $\theta \in \Theta$.

If $\{ \gam_k \}$ is a sequence in $ L^0(\Theta,\cP(\R^n))$, we recall $ \gam_k$ converges weakly to $\gam \in L^0(\Theta,\cP(\R^n))$ if and only if the measures $\dd (\gam_k)_{\theta}(y) \dd \sigma(\theta)$ converge weakly on $\Theta \times \R^n$ to $\dd \gam_{\theta}(y) \dd \sigma(\theta)$.

\begin{lem} \label{l:ap2} 
	For $(\Theta,\sigma)$ a Polish space endowed with a finite measure, 
	let $\phi(\theta,y)$ be a bounded Carath\'eodory function on $\Theta \times \R^n$, and $\{ \gam_k \} \in L^0(\Theta,\cP(\R^n))$ converging weakly to some $\gam \in L^0(\Theta,\cP(\R^n))$.
	Then 
	\begin{equation} 
		\label{e:limitint}
		\lim_{k \to + \infty} \iint_{\Theta \times \R^n} \phi(\theta,y) \dd (\gam_k)_{\theta}(y) \dd \sigma(\theta) = \iint_{\Theta \times \R^n} \phi(\theta,y) \dd \gam_{\theta}(y) \dd \sigma(\theta).
	\end{equation}
\end{lem}

\begin{proof}
	We write $\tplan_k = \dd (\gam_k)_\theta \dd \sigma$ and $\tplan = \dd \gam_\theta \dd \sigma$.
	Fix $\varepsilon > 0$. By the Scorza-Dragoni theorem (see, e.g., \cite[Lemma 4.6]{giusti2003direct}), there exists $K \subset \Theta$ a compact set such that $\phi$ is continuous over $K \times \R^n$ and $\sigma(\Theta \backslash K) \leq \varepsilon$.
	By the extension of Tietze's theorem \cite[Theorem 4.1]{Dugundji1951} (applied to $\phi$ seen as a function from $K$ to $(C_{b}(\R^n,\R), \| \cdot \|_\infty$)), we can find $\tilde{\phi} : \Theta \times \R^n \to \R$ jointly continuous and bounded by $\| \phi \|_\infty$ such that $\tilde{\phi} = \phi$ on $K \times \R^n$. We obtain
	\begin{equation*}
		\left| \iint_{\Theta \times \R^n} \phi \dd \tplan_k - \iint_{\Theta \times \R^n} \phi \dd \tplan \right| \leq 4 \varepsilon \| \phi \|_\infty  + \left| \iint_{\Theta \times \R^n} \tilde{\phi} \, \dd (\tplan_k - \tplan)  \right|.
	\end{equation*}
	Thus as we send $k \to + \infty$, by continuity of $\tilde{\phi}$ the integral in the right-hand side vanishes. It means $\limsup_k \left| \iint_\QQ \phi \dd \tplan_k - \iint_\QQ \phi \dd \tplan \right| \leq 4 \varepsilon \| \phi \|_\infty$, and the result follows since $\varepsilon$ is arbitrary.
\end{proof}

\begin{cor} 
	\label{c:b1} 
	For $(\Theta,\sigma)$ a Polish space endowed with a finite measure, 
	let $\phi(\theta,y)$ be a Carath\'eodory function on $\Theta \times \R^n$ bounded from below, and $\{ \gam_k \} \in L^0(\Theta,\cP(\R^n))$ be converging weakly to $\gam \in L^0(\Theta,\cP(\R^n))$.
	Then 
	\begin{equation*} 
		\liminf_{k \to + \infty} \iint_{\Theta \times \R^n} \phi(\theta,y) \dd (\gam_k)_{\theta}(y) \dd \sigma(\theta) \geq \iint_{\Theta \times \R^n} \phi(\theta,y) \dd \gam_{\theta}(y) \dd \sigma(\theta).
	\end{equation*}
\end{cor} 

\begin{proof} 
	We use the same notations $\tplan_k$ and $\tplan$ as in the proof of Lemma~\ref{l:ap2}.
	Without loss of generality, the $\liminf$ in the left-hand side is finite since otherwise there is nothing to prove. Let $\phi_m \coloneqq \min (\phi,m), m\in \mathbb{N}$.  Then 
	$$\liminf_{k \to \infty}\iint_{\Theta \times \R^n} \phi \dd \tplan_k \ge\lim_{k \to \infty}\iint_{\Theta \times \R^n} \phi_m \dd \tplan_k = \iint_{\Theta \times \R^n} \phi_m \dd \tplan $$ 
	by Lemma \ref{l:ap2}. The result follows by the monotone convergence theorem. \end{proof}

\begin{cor}
	\label{c:b2}
	For $(\Theta,\sigma)$ a Polish space endowed with a finite measure, 
	let $\phi(\theta,y)$ be a Carath\'eodory function on $\Theta \times \R^n$ which is $p$-bounded from below for $p \geq 0$ in the sense of Definition~\ref{defi:p_bounded_coercive}. Let $\{ \gam_k \} \in L^0(\Theta,\cP(\R^n))$ converging weakly to $\gam \in L^0(\Theta,\cP(\R^n))$ be such that, for some $q > p$,
	\begin{equation}
		\label{eq:sup_q_integral}
		\sup_{k} \iint_\QQ |y|^q \dd (\gam_k)_{\theta}(y) \dd \sigma(\theta) < + \infty.
	\end{equation}
	Then we can conclude
	\begin{equation*} 
		\liminf_{k \to + \infty} \iint_{\Theta \times \R^n} \phi(\theta,y) \dd (\gam_k)_{\theta}(y) \dd \sigma(\theta) \geq \iint_{\Theta \times \R^n} \phi(\theta,y) \dd \gam_{\theta}(y) \dd \sigma(\theta).
	\end{equation*}
\end{cor} 

\begin{proof}
	We use the same notations $\tplan_k$ and $\tplan$ as in the proof of Lemma~\ref{l:ap2}.
	Let $a \in L^1(\Theta)$ and $C > 0$ be such that $\tilde{\phi}(x,y) = \phi(x,y) + a(x) + C |y|^p$ is non-negative. Note that $\tilde{\phi}$ is still a Carath\'eodory function. From Corollary~\ref{c:b1}, 
	\begin{equation*}
		\liminf_{k \to + \infty} \iint_{\Theta \times \R^n} \tilde{\phi} \dd \tplan_k \geq \iint_{\Theta \times \R^n} \tilde{\phi} \dd \tplan.
	\end{equation*}
	Expanding the definition of $\tilde{\phi}$, noting that the integrals of $a$ do not depend on $k$ (as they only depend on the first marginal), we obtain the result provided we justify that 
	\begin{equation*}
		\lim_{k \to \infty}\iint_{\Theta \times \R^n} |y|^p \dd \tplan_k = \iint_{\Theta \times \R^n} |y|^p \dd \tplan.
	\end{equation*}
	The latter claim comes as $\{ \tplan_k \}$ is uniformly $p$-integrable thanks to the bound~\eqref{eq:sup_q_integral}, cf.\ \cite[Eq. (5.1.20)]{AmbrosioGigliSavare08}, and thus we can exchange limit and integral for an integrand that has $p$-growth, cf.\ \cite[Lemma 5.1.7]{AmbrosioGigliSavare08}.
\end{proof}

\paragraph{Acknowledgments.}
The research leading to this project started during a visit of AK and HL to DV in the University of Coimbra, which is warmly acknowledged for the hospitality. No AI tool was used for the conception and mathematical content of this project, except for fixing a few English and typographic errors. 

\def\cprime{$'$} \def\cprime{$'$}

\end{document}